%% file: optinf_for_flows.tex
\documentclass[]{mpi2015-cscpreprint}
\usepackage{amsmath}
\usepackage{amssymb}
\usepackage{amsthm}
\usepackage{graphics} 
\usepackage{epsfig} 
\usepackage{amsmath} 
\usepackage{amssymb}  
\usepackage{algorithm}
\usepackage[noend]{algpseudocode}
\makeatletter
\def\algbackskip{\hskip-\ALG@thistlm}
\makeatother

\usepackage{bm}
\usepackage{framed}

\usepackage{tikz,pgfplots} 
\usepgfplotslibrary{groupplots}

\pgfplotsset{compat=newest} 
\pgfplotsset{plot coordinates/math parser=false} 

\usepackage{color}
\usepackage{caption}

\usepackage{mymacros}      

\input{macrosdefs}

\newtheorem{theorem}{Theorem}[section]

\usepackage{float}
\usepackage{color}
\usepackage{array}
\usepackage{dsfont}
\usepackage{tikz}

\renewcommand{\hat}[1]{\widehat{#1}}
\renewcommand{\tilde}[1]{\widetilde{#1}}

\usetikzlibrary{arrows,decorations.pathmorphing,backgrounds,fit,positioning,shapes.symbols,chains}

\newenvironment{eq}{\everymath {\displaystyle \everymath{ }} \equation}{ \endequation} %

\hypersetup{%
    pdftitle={Operator Inference and Physics-Informed Learning of Low-Dimensional Models for Incompressible Flows},
  pdfauthor={Peter Benner, Pawan Goyal, Jan Heiland, Igor Pontes Duff},
    pdfkeywords={operator inference, nonintrusive model reduction,
    incompressible Navier-Stokes equation}
    }

\renewcommand{\hat}[1]{\widehat{#1}}
\renewcommand{\tilde}[1]{\widetilde{#1}}

\begin{document}

\title{Operator Inference and Physics-Informed Learning of Low-Dimensional
  Models for Incompressible Flows}

\author[1]{Peter~Benner}
\affil[1]{Max Planck Institute for Dynamics of Complex Technical Systems,
  Magdeburg, Germany.\newline
Faculty of Mathematics, Otto von Guericke University Magdeburg, Germany.\authorcr
  Email: \texttt{\href{mailto:benner@mpi-magdeburg.mpg.de}{benner@mpi-magdeburg.mpg.de}},
  ORCID: \texttt{\href{https://orcid.org/0000-0003-3362-4103}%
    {0000-0003-3362-4103}}}

\author[2]{Pawan~Goyal}
\affil[2]{Max Planck Institute for Dynamics of Complex Technical Systems,
  Magdeburg, Germany.\authorcr
  Email: \texttt{\href{mailto:goyalp@mpi-magdeburg.mpg.de}{goyalp@mpi-magdeburg.mpg.de}},
  ORCID: \texttt{\href{https://orcid.org/0000-0003-3072-7780}%
    {0000-0003-3072-7780}}}

\author[3]{Jan~Heiland}
\affil[3]{Max Planck Institute for Dynamics of Complex Technical Systems,
  Magdeburg, Germany. \newline
Faculty of Mathematics, Otto von Guericke University Magdeburg, Germany.
  \authorcr
  Email: \texttt{\href{mailto:heiland@mpi-magdeburg.mpg.de}{heiland@mpi-magdeburg.mpg.de}},
  ORCID: \texttt{\href{https://orcid.org/0000-0003-0228-8522}%
    {0000-0003-0228-8522}}}

\author[4,$\ast$]{Igor~Pontes~Duff}
\affil[4]{Max Planck Institute for Dynamics of Complex Technical Systems,
  Magdeburg, Germany.\authorcr
  Email: \texttt{\href{mailto:pontes@mpi-magdeburg.mpg.de}{pontes@mpi-magdeburg.mpg.de}},
  ORCID: \texttt{\href{https://orcid.org/0000-0001-6433-6142}%
    {0000-0003-3072-7780}}}
\affil[$\ast$]{corresponding author}

\shorttitle{Learning of low-dimensional models for incompressible flows} 
\shortauthor{P. Benner, P. Goyal, J. Heiland, and I. Pontes Duff}








\abstract{
Reduced-order modeling has a long tradition in computational fluid dynamics. The
ever-increasing significance of data for the synthesis of low-order models is
well reflected in the recent successes of data-driven approaches such as
\emph{Dynamic Mode  Decomposition} and \emph{Operator Inference}. 
With this work, we discuss an approach to learning structured low-order models
for incompressible flow from data that can be used for engineering studies such
as control, optimization, and simulation.  
To that end, we utilize the intrinsic structure of the Navier-Stokes equations
for incompressible flows and show that learning dynamics of the velocity and
pressure can be decoupled, thus leading to an efficient operator inference
approach for learning the underlying dynamics of incompressible flows.
Furthermore, we show the operator inference performance in learning low-order
models using two benchmark problems and compare with an intrusive method, namely
proper orthogonal decomposition, and other data-driven approaches. 
}



\maketitle
\input{sec-introduction}

\input{sec-ns-reformulation}
\input{sec-projection-based-mor}

\input{sec-pod-basis}

\input{sec-data-based-mor}

\input{sec-operator-inference-nse}
\input{sec-inhomo-conti}

\input{sec-numerical-examples}

\input{sec-conclusion}

\section*{Acknowledgments}

This work was supported by the German Research Foundation (DFG) priority program
1897: "Calm, Smooth and Smart - Novel Approaches for Influencing Vibrations by
Means of Deliberately Introduced Dissipation" 
and the German Research Foundation
(DFG) research training group 2297 "MathCoRe", Magdeburg.
%



\input{optinf_for_flows.bbl}
\typeout{get arXiv to do 4 passes: Label(s) may have changed. Rerun}
\end{document}

%% file: macrosdefs.tex
\newtheorem{remark}{Remark}

\newtheorem{assum}{Assumption}

\newcommand{\bhX}{\mathbf{\hat{X}}}
\newcommand{\bhD}{\mathbf{\hat{D}}}

\newcommand{\bhH}{\mathbf{\hat{H}}}

\newcommand{\bhv}{\mathbf{\hat{v}}}
\newcommand{\btv}{\mathbf{\tilde{v}}}
\newcommand{\btp}{\mathbf{\tilde{p}}}

\newcommand{\dbtv}{\mathbf{\dot{\tilde{v}}}}
\newcommand{\bhp}{\mathbf{\hat{p}}}

\newcommand{\bhA}{\mathbf{\hat{A}}}
\newcommand{\btA}{\mathbf{\tilde{A}}}
\newcommand{\bhB}{\mathbf{\hat{B}}}
\newcommand{\btB}{\mathbf{\tilde{B}}}

\newcommand{\bhV}{\mathbf{\hat{V}}}
\newcommand{\bhP}{\mathbf{\hat{P}}}
\newcommand{\dbhV}{\mathbf{\dot{\hat{V}}}}

\newcommand{\efac}{\mathbf L_{\bE_{11}}}
\newcommand{\efact}{\mathbf L_{\bE_{11}}^\top}
\newcommand{\efacmt}{\mathbf L_{E_{11}}^{-\top}}

\def\Aoo{\mathbf{A}_{11}}
\def\Aot{\mathbf{A}_{12}}

\def\Aott{\Aot^\top}

\def\tAoo{\tilde{\mathbf{A}}_{11}}

\def\tAot{\tilde{\mathbf{A}}_{12}}

\def \tAott{\tAot^{\top}}
\def\tAoo{\tilde{\mathbf{A}}_{11}}
\def\tEoo{\tilde{\mathbf{E}}_{11}}
\def\btH{\tilde{\mathbf{H}}}
\def\Eoo{\mathbf{E}_{11}}

\def\hBo{\mathbf{\hat{B}}_{1}}
\def\tBo{\mathbf{\tilde{B}}_{1}}

\def\fo{\mathbf{f}}
\def\ft{\mathbf{g}}
\def\vz{\mathbf{v}_\top}
\def\dvz{\dot{\mathbf{v}}_\top}

\def\hvz{\hat{\mathbf{v}}}

\def\vp{\mathbf{v}_\perp}
\def\dvp{\dot{\mathbf{v}}_\perp}

\def\Vv{\mathbf{V}_\bv}
\def\tVv{\tilde{\mathbf{V}}_\bv}

\def\DC{\texttt{driven cavity}}
\def\CW{\texttt{cylinder wake}}
\def\opinf{\texttt{OpInf}}
\def\opinflin{\texttt{OpInf\textunderscore lin}}
\def\tol{\texttt{tol}}
\def\FOM{\texttt{FOM}}
\def\ROM{\texttt{ROM}}
\def\DMD{\texttt{DMD}}
\def\DMDc{\texttt{DMDc}}
\def\DMDquad{\texttt{DMDquad}}
\def\POD{\texttt{POD}}

\def\totimes{\mathbin{\tilde\otimes}}

\newcommand{\sspan}[1]{\ensuremath{\mathop{\mathrm{span}}\left( #1 \right)}}

\newlength\figureheight
\newlength\figurewidth

%% file: sec-introduction.tex
\section{Introduction} 

The numerical solution of the incompressible Navier-Stokes equation is a demanding task in terms of computational effort and memory requirements.  Reduced-order models can significantly decrease demands while maintaining a certain accuracy. One may well say that flow simulations have been a success story for the application of model reduction schemes as well as a driving force
for the development of new approaches. For example, the popular methods of
\emph{Proper Orthogonal Decomposition} and \emph{Dynamic Mode Decomposition}
seem to have their origins in flow computations; see \cite{morSir87a,morSch10}. On
the other hand, the distinguished nonlinear and differential-algebraic structure
of the semi-discrete incompressible Navier-Stokes has been attractive for the
testing of extensions of standard system-theoretic approaches like Balanced
Truncation (BT), LQG-BT,
moment-matching, or polynomial expansions of the HJB equations; see
\cite{morAhmBGetal17,morBenH15,BreKP19,morHeiSS08,morBenSU16}. 
Over the years, many applications and variants of POD for
incompressible flows have been reported. We only mention particular works that include finite volume schemes, the use of supremizers, and pressure reconstruction \cite{StaHMLR17,StaR18} that adaptively improve the bases for the use in optimal control \cite{Rav00}, or works that address the algebraic constraint in different spatial discretizations \cite{morGraHLU19} and refer to the numerous references therein. For an application of the \emph{Proper Generalized Decomposition} (PGD), see \cite{DumAA11}.


Albeit the success of the projection-based methods such as POD, one of the significant drawbacks of these methods is that they require a discretized full-order model. It may be unavailable or a cumbersome task to obtain it.  One may think of a scenario when a process is simulated using a black-box commercial solver. Hence, data-driven modeling has gained interest, aiming to determine reduced-order models directly using data that may be either obtained using a solver or in an experimental set-up. 
There exist vast literature on learning dynamical systems from data, and we
review the most relevant literature to this work in what follows.  Often
referred to as \emph{system identification}, the inference of models for linear
time-invariant (LTI) systems has a long history with applications in control
theory. One of the most relevant classical approaches is the eigenvalue
realization algorithm, which was first proposed in \cite{HoK66} in the minimal
realization set-up and then in \cite{JuaP85} for system identification. This
algorithm is based on the Hankel matrix using data from the system's impulse
response. For frequency domain data, several approaches were designed allowing
rational interpolation of LTI transfer functions through, e.g., vector-fitting
\cite{GusS99}, Loewner approach \cite{morMayA07,morPehGW17}, and the so-called
\emph{AAA algorithm} \cite{morNakST18}. 

On the other hand, nonlinear system identification requires an assumption on the
structure of the non-linearity. Towards this, the Loewner framework has been
extended to classes of nonlinear systems, especially bilinear and quadratic-bilinear systems in \cite{morAntGI16} and \cite{morGosA18}, respectively. Other identification methods are based on block-oriented nonlinear modeling, which defines systems as interconnections of an LTI block and static nonlinear blocks \cite{giri2010block}. An example of these block-oriented modeling is the Wiener-Hammerstein systems, and many works have been dedicated to the identification of such systems, see, e.g., \cite{billings1982identification, wills2013identification, schoukens2017identification}. Lately, sparse identification has also emerged as a powerful technique \cite{BruPK16}. The approach's principle is to pick up a few nonlinear terms from a vast library of candidate nonlinear functions representing the dynamics of the system. However, it relies on the fact that the nonlinear terms that are important to represent the dynamics should be presented in the library.
Another attractive data-driven method for nonlinear systems is based on the
interpretation of dynamical systems via the \emph{Koopman operator}, which
readily applies to nonlinear systems. The identification is done by the so-called \emph{Dynamic Mode Decomposition} (DMD).
Curiously, DMD has initially been developed for low dimensional approximations
of fluid dynamics \cite{morSch10}. Since then, it has seen various extensions
such as extended DMD \cite{WilKR15}, kernel DMD \cite{WilRK15}, higher-order DMD
\cite{LecV17}. Furthermore, it is intrinsically related to the Koopman operator
analysis, see \cite{Mez13}.  Moreover, DMD can be used to simultaneously
identify an input operator \cite{ProBK16, ProBK18}, and an output operator
\cite{AnnGS16, morBenHM18}.  Additionally in \cite{morGosP20}, the authors
propose the DMD with control incorporation of quadratic-bilinear terms. We refer
readers to \cite{morKutBBetal16} for a comprehensive monograph on this topic.   


One often comes across scenarios where prior knowledge, such as governing equations of the process, is known. In these set-ups, however, the system parameters and discretization scheme are not known, but the underlying partial differential equations that govern the dynamics are known.  Under these assumptions, the operator inference approach has emerged as a powerful tool to learn reduced-order models operators. It was first studied in \cite{morPehW16}, aiming at identifying nonlinear polynomial systems, which was later extended to general nonlinear systems \cite{QKPW2020_lift_and_learn,morBenGKPW20}. 
Moreover, recently, physics-informed neural network-based approaches have received a lot of attention, where one aims to utilize the known knowledge such as partial differential equations to build a model to solve nonlinear multi-physics problems, see, e.g., \cite{raissi2019physics,raissi2018deep,raissi2018multistep}.

In this work,  we focus on constructing models for incompressible flows. The dynamics of such a flow is, typically, given by the Navier-Stokes equations that comprise a differential equation and an algebraic equation. The algebraic equation enforces the incompressibility condition. 
%
Even though one may write down the equations, the discretization scheme and parameters such as the Reynolds number may be unknown. If the Reynolds is known, then physics-informed neural network for incompressible flow can be utilized that explicitly makes use of the known Navier-Stokes equations for incompressible flows, see \cite{jin2020nsfnets}. However, the approach requires system parameters such as Reynolds numbers to be known, which may not be available, and the proposed methodology is valid for uncontrollable flows. Moreover, the proposed neural network has high-dimensional outputs that are the velocity and pressure in a given domain. Hence, the network can be computationally expensive for large-scale flows. But we know that the dynamics of large-scale dynamical systems typically lie in a low-dimensional subspace. With that aim, we propose a scheme to identifying reduced-order models for flows using data that incorporates the knowledge of Navier-Stokes equations for incompressible flows to learn reduced-order models from data; particularly, we discuss how to employ the operator inference approach proposed in \cite{morPehW16} to incompressible flows governed by Navier-Stokes equations.

To this end, we discuss {an application of} the operator inference approach{, proposed in \cite{morPehW16}} to build low-order models for incompressible flows.
For this, we use the intrinsic structure of the semi-discrete Navier-Stokes
equations and show that the learning dynamics of the velocity and pressure can
be decoupled. As the main result, we present an efficient learning algorithm for
modeling incompressible flows via a low-dimensional ordinary differential
equation (ODE).  In practice, it is a nontrivial task %
to extract such an \emph{underlying ODE} from the model equations. As we will discuss, standard projection based approaches like
POD may require corrections of the model and the projection basis and still suffer from systematic model errors caused by numerical inaccuracies.  {On the other hand,} the tailored operator inference method {for incompressible flows} directly identifies an \emph{underlying
ODE} without being subjected to these theoretical and practical issues. Moreover, we show that in theory and under certain assumptions, the intrusive POD model converges to the learned model using the proposed non-intrusive method. 
%

The rest of the manuscript is structured as follows. In Section
\ref{sec:nse-dae-ode}, we introduce the spatially discretized incompressible
Navier-Stokes equations in the velocity-pressure formulation and recall how to
write the underlying ordinary differential equation for the velocity alone by
eliminating the coupling with the pressure. In \Cref{sec:projection-mor}, we
discuss an intrusive method (POD) to determine reduced-order models. In
\Cref{sec:data-based-mor}, we recap the operator
inference approach, proposed in \cite{morPehW16}.  In \Cref{sec:opinf-nse}, we  tailor the operator inference
approach for incompressible flows and show that intrusive and non-intrusive
models converge under certain conditions. In Section \ref{sec:ftwo-not-zero}, we
discuss the case of inhomogeneity in the incompressibility constraints and its
implications for the model reduction schemes. Finally, we present illustrative numerical studies in Section \ref{sec:num-results}, and in the subsequent section, we conclude the paper with summarizing remarks and an outlook on future related research.


%% file: sec-ns-reformulation.tex
\section{Navier-Stokes  and its Equivalent Transformation as ODE}\label{sec:nse-dae-ode}
In this paper, we focus on the Navier-Stokes equations given in the partial differential form as follows:
\begin{equation}\label{eq:pDE_NS}
	\begin{aligned}
		\dfrac{\partial \mathfrak{v}}{\partial t} &= - \mathfrak{v}\nabla\mathfrak{v} + \dfrac{1}{\texttt{Re}} \Delta\mathfrak{v} - \nabla \mathfrak{p} + f \quad &  &\text{on}~\Omega, \\
		\texttt{div}~\mathfrak{v} &= g & &\text{on}~\Omega, \\
		\dfrac{\partial \mathfrak{v}}{\partial t} & =0 &&\text{on}~\partial \Omega, \\
		\mathfrak{v} &= \mathfrak{v}_0 && \text{for}~t=0,
	\end{aligned}
\end{equation}
where $\mathfrak v$ and $\mathfrak{p}$  are the velocity and pressure, respectively, and $\Omega$ and $\partial \Omega$ denote the domain and boundary of the domain, respectively; \texttt{Re} denotes the Reynolds number.  We begin by presenting the differential algebraic equations (DAE), arising from the semi-discretization of the Navier-Stokes equation \eqref{eq:pDE_NS}; that is,
\begin{subequations}\label{eq:DisctNavSto}
	\begin{align}
		\Eoo\dot{\bv}(t) &= \Aoo\bv(t) + \Aot\bp(t) + \bH\left(\bv \otimes \bv\right)  +  \mathbf{f}(t), \label{eq:DiffEq} \\ 
		0 &=  \Aott\bv(t) + \bg(t), \quad \bv(0) = \bv_0, \label{eq:AlgEq}  
	\end{align}
\end{subequations}
where $\bv(t) \in \R^{n_\bv}$, $\bp(t) \in \R^{n_\bp}$ are, respectively, the velocity and the pressure with $n_\bv>n_\bp$,  and  $\Eoo, \Aoo \in \R^{n_\bv \times n_\bv}$, $\Aot \in \R^{n_\bv \times n_\bp}$, $\mathbf{f}(t) \in \R^{n_\bv \times 1}$, $\bg (t)\in \R^{n_\bp\times 1}$; the initial condition for the velocity is denoted by $\bv_0$. Furthermore, we assume that $\bE_{11}$ and $\Aott\Eoo^{-1}\Aot$ are nonsingular matrices; hence, the system \eqref{eq:DisctNavSto} is an index-2 linear system when $\bH \equiv 0$. Additionally, we will suppose that the forcing terms are giving by
\begin{equation*}
	\mathbf{f}(t) = \bB_1 \bu(t) \quad \text{and} \quad \bg(t) \equiv 0,
\end{equation*} 
where $\bu(t)\in \R^m$ is an input vector and $\bB_1 \in \R^{n_\bv \times m}$; we show how to handle the case when $\bg(t) \neq 0$ in \Cref{sec:ftwo-not-zero}.

Next, we aim at transforming the DAE system \eqref{eq:DisctNavSto} into an equivalent ODE system, which will be crucial in the analysis done in the rest of the paper. 

\subsection*{Transformation of Navier-Stokes DAEs} Taking the time derivative of the algebraic conditions  $\Aott \bv(t)= \bg(t) = 0$ $\forall t\geq 0$, implies that $\Aott \dot \bv(t)=0$ $\forall t \geq 0$. Then, we  premultiply \eqref{eq:DiffEq} by $\Aott \Eoo^{-1}$ from the left-hand side, yielding:
\begin{equation*}
	0 = \Aott \Eoo^{-1}\Aoo \bv(t) + \Aott \Eoo^{-1}\Aot \bp(t) + \Aott \Eoo^{-1}\bH
	(\bv(t) \otimes \bv(t)) + \Aott \Eoo^{-1}\fo (t).
\end{equation*}
With $\bS:=\Aott \Eoo^{-1} \Aot$ being invertible, the pressure $\bp$ can be
expressed in terms of $\bv$ and $\fo$ as follows:
\begin{equation}\label{eq:def-p}
	\bp(t) = -\bS^{-1}\Aott \Eoo^{-1}(\Aoo \bv(t) + \bH(\bv(t)\otimes \bv(t)) + \fo(t)).
\end{equation}
Substituting $\bp(t)$ in \eqref{eq:DiffEq} from the above equation, we obtain an ODE for the velocity as follows:
\begin{equation}\label{eq:nse-ode}
	\Eoo \dot \bv(t) = \Pi^\top\Aoo \bv(t) + \Pi^\top\bH(\bv(t)\otimes \bv(t)) +
	\Pi^\top\fo(t),
\end{equation}
where
\begin{equation}\label{eq:def-disc-leray}
	\Pi^\top = \bI - \Aot \bS \Aott \Eoo^{-1}.
\end{equation}
This leads to the important observation that although both the velocity and pressure are evolved and are coupled through the DAE \eqref{eq:DisctNavSto}, the DAE can be decoupled into an algebraic and differential part. Precisely, the evolution of the velocity over time can be given by the quadratic ODE \eqref{eq:nse-ode}, and there exists an algebraic equation that links the pressure and velocity given by \eqref{eq:def-p}. 

\begin{remark}
	Equation \eqref{eq:def-p} is often referred to as \emph{Pressure Poisson
		Equation}, and the projector defined in \eqref{eq:def-disc-leray} is called
	the \emph{discrete Leray Projector}.
\end{remark}

%% file: sec-projection-based-mor.tex
\section{Projection-based Reduced-order Modeling for Navier-Stokes Equations}\label{sec:projection-mor}
In this section, we discuss the construction of reduced-order models as surrogates for the  large-scale DAEs in \eqref{eq:DisctNavSto}.  To preserve the DAE structure, we consider block Galerkin projections. To that end, we aim at finding projection matrices $\bV_\bv\in \R^{n_\bv \times r_\bv}$ and $\bV_\bp\in \R^{n_\bp \times r_\bp}$ such that
\[\bv(t) \approx \bV_\bv \btv(t) \quad \text{and} \quad \bp \approx \bV_\bp \btp(t). \] 
This leads to a reduced-order model that preserves the DAE structure:
\begin{subequations}\label{eq:ROMNavSto}
	\begin{align}
	\tEoo \dbtv(t) &=  \tAoo\btv(t) + \tAot\btp(t) + \btH\left(\btv \otimes \btv\right)  + \btB_1 \bu(t)  , \label{eq:Red_DiffEq} \\ 
	0 &=  \tAott \btv(t), \label{eq:Red_AlgEq}  
	\end{align}
\end{subequations}
where
\begin{align}\label{eq:pod_red_matrices}
\tEoo &= \bV_\bv^\top\bE_{11}\bV_\bv, & \tAoo &= \bV^\top_\bv\bA_{11}\bV_\bv, \quad \tAot = \bV_\bv^\top\bA_{12}\bV_\bp,\\  
\btH &= \bV_\bv^\top\bH\left(\bV_\bv \otimes \bV_\bv\right), &   \text{and}  \quad \btB_1 &=\bV_\bv^\top\bB_1.
\end{align}
The above reduced-order system can be  obtained in a Galerkin projection framework using a holistic projection matrix as follows:
\[\bV = \begin{bmatrix}
\bV_\bv & \\
& \bV_\bp
\end{bmatrix}. \]

%% file: sec-pod-basis.tex
\subsection*{POD projection bases}
There exist several approaches to determine the projection matrices, see, e.g., \cite{morAhmBGetal17,morBenH15,BreKP19,morHeiSS08}. Here, we consider the case when the bases are determined using POD.  In its basic form,
an $r_\bv$-dimensional $\bV_\bv$  projection matrix is determined using the $r_\bv$ principal singular vectors of the snapshot matrix:
\begin{equation}
	\mathbb V = 
	\begin{bmatrix}
		\bv(t_0), \bv(t_1), \ldots,\bv(t_N)
	\end{bmatrix}.
\end{equation}
One of many equivalent characterizations is the choice that minimizes the index
\begin{equation}\label{eqn:podbas-index}
	\|\mathbb V - \Pi_{r_\bv} \mathbb V\|_F
\end{equation}
over all projections $\Pi_{r_\bv}$ of rank $k$ with $\bV_\bv\bV_\bv^\top$ being the
minimizing projection.

In the context of the system \eqref{eq:DisctNavSto}, arising from  the FEM or FVM discretization of the Navier-Stokes equations, the matrix $\bE_{11}$ is a symmetric positive definite \emph{mass matrix} that induces the \emph{discrete $L_2$ norm} via
\begin{equation}
	\|\bv(t)\|_{\bE_{11}} := \sqrt{\bv(t)^\top\bE_{11}\bv(t)}
\end{equation}
and the corresponding index \eqref{eqn:podbas-index} as
\begin{equation}\label{eqn:podbas-index-wrtm}
	\|\efact(\mathbb V - \Pi_k \mathbb V)\|_F,
\end{equation}
where $\efac$ is a factor of $\Eoo$, i.e., $\efac\efact=\Eoo$.
Hence, the $r_\bv$-dimensional POD basis that minimizes the index
\eqref{eqn:podbas-index-wrtm} and respects the mass
matrix, is given as 
\begin{equation}
	\bV_\bv := \efact \tilde \bV_{r_\bv},
\end{equation}
where $\tilde \bV_{r_\bv}$ contains  the $r_\bv$ dominant left singular vectors of the matrix
\begin{equation*}
	\efacmt \mathbb V;
\end{equation*}
cp. \cite[Lem. 2.5]{morBauBH18}.

\begin{remark}
	In the true divergence-free case, i.e., if in the algebraic constraint
	\eqref{eq:AlgEq} $\bg \equiv 0$, then the solution snapshots and 
	the POD modes fulfill the constraint so that 
	\[ \Aott \mathbb V = 0 \quad \textnormal{and} \quad \Aott \Vv=0\] and, thus, the reduced
	coefficients read $\tAot=0$. Accordingly, the projected
	equations decouple and reduce to \eqref{eq:Red_DiffEq}. As a consequence, the reduced-order model obtained by POD has the following ordinary differential equation structure
	\begin{eq}
		\tEoo \dbtv(t) = \tAoo\btv(t)  + \btH\left(\btv \otimes \btv\right)  +   \tBo\bu(t), \label{eq:red_pod_divfree} 
	\end{eq}
	where the algebraic conditions are automatically satisfied by the POD basis $\bV_\bv$.  
\end{remark}

\begin{remark}\label{rem:POD_algcond}
	Even in the case when the snapshots are divergence-free, it may happen that the POD modes will not satisfy the constraints uniformly well as standard computational approaches for an SVD focus on the decomposition aspect. For a uniform accuracy of the decomposition, the accuracy of a singular vector can be inversely proportional to the associated singular value. 
	For ODE systems, this inaccuracy is not an issue since the dominant modes are computed with high precision. For DAE systems, however, a failure in complying with the constraints introduces a systematic error. In the numerical examples later in the paper, we illustrate how the POD modes deviate from the actual subspace and how this limits the accuracy of POD for incompressible Navier-Stokes equations.
  A remedy for this could be the use of highly-accurate SVD computations
  \cite{DrmV07} as it is available in LAPACK. 
\end{remark}

%% file: sec-data-based-mor.tex
\section{Operator Inference Approach}\label{sec:data-based-mor}
In the previous section, we discussed the projection-based model reduction based on
POD. The approach requires the knowledge of the (semi-discretized in space)
system matrices in \eqref{eq:DisctNavSto}, the so-called system realization. The
reduced basis is obtained by compressing the time-domain snapshots of the
system. Hence, the reduced-order model is typically constructed in the framework
of Galerkin projection. Although this approach is effective in many set-ups, the assumption of having the full order model realization might lead to some practical limitations. In various instances, the dynamical system  or a partial differential equation solver works as a black-box software, and the user has no access to the system realization. Hence, the authors in \cite{morPehW16} have proposed an operator inference framework that aims at constructing reduced-order models directly using data without having access to the realization explicitly. 
In the following, we first give a brief overview of the operator inference approach\cite{morPehW16} for ODEs.  Let us consider a system of the form
\begin{equation}\label{eq:quad_model}
	\dot \bx(t) = \bA\bx(t) +  \bH\left(\bx(t)\otimes \bx(t)\right) +  \bB\bu(t), \quad \bx(0) = \bx_0,
\end{equation}
where $\bx(t) \in \Rn$, $\bu(t) \in \Rm$, and all the other system matrices are of appropriate size.  We are interested in constructing a reduced-order model of the form:
\begin{equation}\label{eq:optinf_rom}
	\dot{\hat\bx}(t) = \hat\bA\hat\bx(t) +  \hat\bH\left(\hat\bx(t)\otimes \hat\bx(t)\right) +  \hat\bB\bu(t), \quad \hat\bx(0) = \hat\bx_0,
\end{equation}
where $\hat\bx(t) \in \Rr$ such that $\bx(t) \approx \bV \hat\bx(t)$, where $\bV$ is a projection matrix, and all the other system matrices are of appropriate size.  We aim at constructing the reduced-order model \eqref{eq:optinf_rom} using the snapshots of $\bx(t)$  at time steps $t_0, t_1,\ldots,t_N$ and input snapshots at the same time steps, i.e., $\bu(t_0),\ldots, \bu(t_N)$. Now, we first collect these snapshots in matrices as follows:
\begin{equation}
	\bX = \begin{bmatrix} \bx_0,\ldots, \bx_N \end{bmatrix},\quad 	\bU = \begin{bmatrix} \bu_0,\ldots, \bu_N \end{bmatrix},
\end{equation}
where $\bx_i := \bx(t_i)$ and $\bu_i := \bu(t_i)$.  Next, we determine dominant POD bases using the SVD of the matrix $\bX$ and construct the projection matrix, denoted by $\bV \in \Rnr$ using the $r$ dominant basis vectors. This allows us to  compute the  reduced state trajectory $\hat\bX = \bV^\top \bX$, where 
$$\hat\bX := \begin{bmatrix} \hat\bx(t_0), \ldots, \hat\bx(t_N)\end{bmatrix}  \quad\text{with}\quad \hat\bx(t_i)= \bV^\top \bx_i.$$
Furthermore, we denote the derivative of $\hat \bx(t)$ at the time steps
$t_1\ldots,t_N$ by $\dot{\hat\bx}(t_0),\ldots, \dot{\hat\bx}(t_N)$, which can be
approximated using the reduced trajectory $\hat \bx(t)$, see, e.g.,
\cite{nguyen2008best,MarH13}. Next, we collect the time-derivative approximation of $\hat\bx(t)$ in a matrix as:
\begin{equation}
	\dot{\hat\bX} = \begin{bmatrix} 		\dot{\hat\bx}(t_0), \ldots,\dot{\hat\bx}(t_N)	\end{bmatrix}.
\end{equation}
We would like to point out that there exist several commercial software that
make us available $\dot \bx(t)$ as well. In such a case,  an alternative way to
compute the time-derivative of $\hat \bx(t)$ is by projecting $\dot \bx(t)$
using the projection matrix $\bV$, i.e., $\dot{\hat\bx}(t) = \bV^\top \dot
\bx(t)$. However, in this paper, we consider the case where we do not have
access to $\dot \bx(t)$ and aim at estimating it using the data of $\hat\bx(t)$.

Having the projected data, we solve the following least-squares optimization problem to determine the operators of the reduced-order model \eqref{eq:optinf_rom}:
\begin{equation}\label{eq:opinf_optimization}
	\min_{\hat\bA,\hat\bH,\hat\bB} \left\|\dot{\hat\bX} - \begin{bmatrix}\hat \bA, \hat\bH, \hat\bB \end{bmatrix} \cD \right\|,
\end{equation}
where $\cD = \begin{bmatrix} \hat\bX \\ \hat\bX\totimes \hat\bX \\ \bU \end{bmatrix}$,  and the product $\totimes$ is defined as
\begin{equation}
	G\totimes G = \begin{bmatrix}		g_1\otimes g_1,\ldots, g_N\otimes g_N 	\end{bmatrix}
\end{equation}
with $g_i$ being $i$-th column of the matrix $G\in \R^{n\times N}$. In many cases, we may find the matrix $\cD$ to be ill-conditioned. There exist various possible ways to circumvent this issue, see, e.g., \cite{yildiz2020data,mcquarrie2020data}.  In this paper, we use an SVD based approach; a similar approach is used in the context of DMD as well, e.g., in \cite{bai2020dynamic}. For this, we consider first the SVD of the matrix $\cD$, denoted as follows:
\begin{equation}
	\cD = \cU \Sigma\cV^\top
\end{equation}
Next, we neglect the singular values of the matrix $\cD$ smaller than a given tolerance; hence, we can write
\begin{equation}
	\cD \approx \tilde\cU \tilde\Sigma \tilde\cV^\top,
\end{equation}
where $\tilde\cU \in\R^{n+n^2 + m\times r}$, $\tilde\Sigma \in \Rrr$ and $\tilde\cV \in \R^{N\times r}$ and the number of singular values, larger than the tolerance, is $r$. As a result, we determine the matrices $\hat\bA,\hat\bH,\hat\bB$ as follows:
\begin{equation}
	\begin{bmatrix}
		\hat\bA,\hat\bH,\hat\bB 
	\end{bmatrix} =  \dot{\hat\bX} \tilde\cV\tilde\Sigma^{-1}\tilde\cU^\top.
\end{equation}
We summarize all the necessary steps to obtain reduced-order models using the operator inference approach in \Cref{algo:opinf} with a slight modification, which is a regularization step.  

\begin{algorithm}[tb!]
	\caption{Operator inference approach to learn low-dimensional models using data \cite{morPehW16}.}
	\hspace{\algorithmicindent} \textbf{Input:} Snapshots $\bx(t)$ and $\bu(t)$ at time steps $t_0, t_1,\ldots,t_N$, \texttt{tol}.
	\begin{algorithmic}[1]
		\State Determine projection matrix $\bV \in \Rnr$, containing the dominant POD modes (or dominant singular vectors)  of the matrix $\bX = \begin{bmatrix} \bx(t_0),\ldots, \bx(t_N) \end{bmatrix}$.
		\State Compute the reduced trajectories $\hat\bx(t)$ at the same time steps:
		$$ \hat\bX := \begin{bmatrix}
  \hat\bx(t_0),\ldots, \hat\bx(t_N)
		\end{bmatrix},$$
		where $\hat\bx(t_i) = \bV^\top\bx(t_i), i \in \{0,\ldots,N\}$. 
		\State Approximate the time-derivative of $\hat\bx(t)$ using the reduced trajectory and collect them in a matrix as follows:
		\begin{equation*}
			\dot{\hat\bX} := \begin{bmatrix}\dot{\hat\bx}(t_0), \ldots, \dot{\hat\bx}(t_N) \end{bmatrix}.  
		\end{equation*}
		\State Compute the SVD of the matrix
		\begin{equation*}
			\cD = \begin{bmatrix} \hat\bX \\ \hat\bX\totimes \hat\bX  \\ \bU \end{bmatrix} =: \cU\Sigma\cV^T,
		\end{equation*}
	where $\Sigma$ contains the singular values of the matrix $\cD$ in the descending order. 
		\State Consider the singular values of the matrix $\cD$ larger than $\tol$. Let the number of singular values larger than $\tol$ be $\tr$.
		\State Compute the operators of the reduced-order model \eqref{eq:optinf_rom} as follows:
		\begin{equation*}
			\begin{bmatrix}\hat\bA,\hat\bH,\hat\bB \end{bmatrix} = \dot{\hat\bX}\tilde\cV^T\tilde\Sigma^{-1}\tilde\cU^T,
		\end{equation*}
		where $\tilde\cU$ and $\tilde\cV$ are the first $\tr$ columns of $\cU$ and $\cV$, respectively, and  $\tilde\Sigma$ is the principal $\tr\times \tr$ block of the matrix $\Sigma$. 
	\end{algorithmic}
	\hspace*{\algorithmicindent} \textbf{Output:} Operators:
	$\hat\bA,\hat\bH,\hat\bB$ of a reduced-order model, having the structure as
	\eqref{eq:optinf_rom}. 
	\label{algo:opinf}
\end{algorithm}


\begin{remark}\label{rem:l-curve} 
	The tolerance, as an input to \Cref{algo:opinf},  is a hyper-parameter. A larger tolerance yields a larger 
	mismatch of the data fidelity term, whereas, for a smaller tolerance, the problem will become ill-conditioned. Therefore, one can employ the idea of an $L$-curve, discussed in \cite{Han00} in the context of computed tomography. The $L$-curve analysis allows us to determine good tolerance that has a good compromise between the matching of data-fidelity term and making the problem well-conditioned. 
\end{remark}

%% file: sec-operator-inference-nse.tex
\section{Operator Inference  for Navier-Stokes Equations}\label{sec:opinf-nse}
This section tailors the operator inference framework presented in the previous section to flow problems. We assume that the flow is incompressible, and the Navier-Stokes equations govern the underlying dynamics \eqref{eq:DisctNavSto}. 
%
%
With this assumption, consider the snapshots of the velocity and pressure vectors at the time steps $t_0, t_1, \dots, t_N$ and associate snapshot matrices as follows:
\begin{equation}\label{eq:vel_press_snapshots}
	\bV = \begin{bmatrix}
		\bv(t_0), \bv(t_1), \dots, \bv(t_N)
	\end{bmatrix}\quad \text{and} \quad \bP = \begin{bmatrix}
		\bp(t_0), \bp(t_1), \dots, \bp(t_N)
	\end{bmatrix} 
\end{equation}
with the corresponding input matrix 
\[\bU = \begin{bmatrix}
	\bu(t_0), \bu(t_1), \dots, \bu(t_N)
\end{bmatrix}.\]
Our goal is to construct reduced-order models that can predict the evolution of the velocity and pressure over time. We focus on constructing such models using only the velocity and pressures snapshot matrices $\bV$ and $\bP$ defined in \eqref{eq:vel_press_snapshots} without having access to the discretized system \eqref{eq:DisctNavSto}  and its matrix realization.   With this aim, let $\bV_\bv$ and $\bV_\bp$ be the dominant POD basis vectors of the velocity and pressure snapshot matrices $\bV$ and $\bP$, respectively. As a consequence, we can define the projected reduced velocity and pressure trajectories as follows:
\begin{align}
	\bhV &= \begin{bmatrix}
		\bhv(t_0) & \bhv(t_1), \dots,\bhv(t_N)
	\end{bmatrix}  = \bV_\bv^\top\bV,~~ \text{and} \label{eq:red_vel}\\
	\bhP &= \begin{bmatrix}
		\bhp(t_0) & \bhp(t_1),\dots, \bhp(t_N)
	\end{bmatrix} = \bV_\bp^\top\bP.
\end{align}
Moreover, let $\dot{\hat \bv}(t_i)$ denote the derivative of the reduced velocity $\hat \bv(t)$ at time $t_i$. It can be approximated using the reduced velocity trajectory $\hat\bv$ as given in \eqref{eq:red_vel}. We collect the derivative information of $\bv(t)$ in a matrix:
\begin{equation}
	\dbhV = \begin{bmatrix}
		\dot{\hat \bv}(t_0),\ldots, \dot{\hat \bv}(t_N)
	\end{bmatrix}.
\end{equation}

Since the original system, arising from the Navier-Stokes discretization, has a quadratic DAE structure as \eqref{eq:DisctNavSto}, one could blindly consider  a least-squares problem that yields a reduced DAE model \eqref{eq:ROMNavSto}, having  the same structure as in \eqref{eq:DisctNavSto}. The matrices involved in defining \eqref{eq:ROMNavSto} are the solution of the following optimization problem:
\begin{equation*}
	\min\limits_{\bhA_{11}, \bhA_{12}, \bhH , \bhB_1} \left\| \begin{bmatrix}
		\bhA_{11} & \bhA_{12} & \bhH & \bhB_1  \\
		\bhA_{12}^\top & 0 & 0 & 0 
	\end{bmatrix}\bhX -\bhD  \right\|^2_{F},
\end{equation*} 
where \[ \bhX = \begin{bmatrix}
	\bhV \\
	\bhP \\
	\bhV {\totimes} \bhV \\
	\bU
\end{bmatrix} \quad \text{and} \quad \bhD = \begin{bmatrix}
	\dbhV\\ 
	0
\end{bmatrix}.  \]
Note that so far, we have not used any additional information on the
Navier-Stokes equations. One of the important observations we make is that the
dynamics of the velocity in the Navier-Stokes equations is given independently of the
pressure. Precisely, there is an underlying ODE for the velocity that completely describes
its dynamics, see \eqref{eq:nse-ode}. This can also be seen when a reduced-order
model is constructed using POD as described in \Cref{sec:projection-mor} as long
as the choice of the POD basis is divergence-free, i.e.,  $\Aott \bV_\bv = 0$. Moreover, the pressure and velocity are connected via an algebraic equation as shown in \eqref{eq:def-p}. Hence, it is conceivable to find an ODE for the reduced velocity $\hat\bv(t)$   as follows:
\begin{equation}\label{eq:opinf_red_vel}
	\dot{\hat\bv}(t) = \hat\bA_\bv\bv(t) + \hat\bH_\bv\left(\bv(t)\otimes \bv(t)\right) + \hat\bB_\bv\bu(t),
\end{equation}
and determine the algebraic equation relating pressure and velocity in a reduced space of the form
\begin{equation}\label{eq:opinf_red_pressure}
	\hat\bp(t) = \hat\bA_\bp\bv(t) + \hat\bH_\bp\left(\bv(t)\otimes \bv(t)\right) + \hat\bB_\bp\bu(t).
\end{equation}

To that end, we can determine the involved operator to define \eqref{eq:opinf_red_vel} by the classical operator inference described in Section \ref{sec:data-based-mor}. For this, we aim at solving the following optimization problem:
\begin{equation}\label{eq:opinf_optimization_vel}
	\min_{\hat\bA_\bv,\hat\bH_\bv,\hat\bB_\bv} \left\|\dot{\hat\bV} - \begin{bmatrix}\hat \bA_\bv, \hat\bH_\bv, \hat\bB_\bv \end{bmatrix} \cD \right\|_F,
\end{equation}
where $\cD = \begin{bmatrix}  \hat\bV \\ \hat\bV\totimes\hat\bV \\ \bU \end{bmatrix}$. Similarly, we can solve the following optimization problem to find the operators of \eqref{eq:opinf_optimization_pre}:
\begin{equation}\label{eq:opinf_optimization_pre}
	\min_{\hat\bA_\bp,\hat\bH_\bp,\hat\bB_\bp} \left\|{\hat\bP} - \begin{bmatrix}\hat \bA_\bp, \hat\bH_\bp, \hat\bB_\bp \end{bmatrix} \cD \right\|_F.
\end{equation}
As discussed in \Cref{sec:data-based-mor}, the optimization problems \eqref{eq:opinf_optimization_vel} and \eqref{eq:opinf_optimization_pre} can be ill-conditioned; thus, we employ the similar remedy based on the SVD as discussed therein. 

Next, we analyze the non-intrusive reduced model \eqref{eq:opinf_red_vel} with
respect to the one obtained using an intrusive model. For this, we extend the
result from \cite[Sec 3.2]{morPehW16} and show that the non-intrusive model \eqref{eq:opinf_red_vel} converges to the intrusive one. For this, we  first make the following four assumptions:


\begin{assum}\label{assum:scheme_conv} The time-stepping scheme for the discretized model  \eqref{eq:DisctNavSto} is convergent in the $L_2$ norm for $\bv$, i.e.,
	\[ \max_{i \in \{1, \dots, T/\Delta t\} }\| \bv_i - \bv(t_i)\|_2 \rightarrow 0 \quad \text{as} \quad \Delta t\rightarrow 0.  \]	 
\end{assum}	

\begin{assum}\label{assum:deriv_conv} The derivative approximation  from the projected state converges to $\frac{d}{dt}\bhv(t_N)$ as $\Delta t \rightarrow 0$.
\end{assum}	

\begin{assum}\label{assum:fulrank_data} The matrix data $\cD = \begin{bmatrix}  \hat\bV \\ \hat\bV\totimes\hat\bV \\ \bU \end{bmatrix}$ has full column rank.
\end{assum}	

\begin{assum}\label{assum:alg_con}  $\bV_\bv \in \R^{n \times r}$ is a divergence-free projected basis, i.e., $\Aott \bV_\bv = 0$.
\end{assum}	

\begin{theorem}\label{thm:POD_opinf} Let Assumptions \ref{assum:scheme_conv}, \ref{assum:deriv_conv}, \ref{assum:fulrank_data}, and \ref{assum:alg_con} hold. Let $\tEoo$, $\tAoo$, $\btH$  and $\tBo$ be the reduced order ODE obtained by Galerkin projection as in \eqref{eq:red_pod_divfree} using the basis $\bV_\bv$.  Then, for every $\varepsilon>0$, there exist an $r<n_\bv$ and a $\Delta t>0$, such that
	\begin{equation}
		\|\tEoo^{-1}\tAoo - \bhA\|_F < \varepsilon, \quad  \|\tEoo^{-1}\btH - \bhH\|_F < \varepsilon, \quad \text{and}, \quad      \|\tEoo^{-1}\tBo - \hBo\|_F < \varepsilon.
	\end{equation}
\end{theorem}

\begin{proof} 
	The proof follows the same lines as the one for  \cite[Theorem 1]{morPehW16}.
  From Assumption \ref{assum:alg_con}, we know that the projected reduced model is of the quadratic ODE form \eqref{eq:red_pod_divfree}, or equivalently, 
	\begin{eq}
		\dbtv(t) = \tEoo^{-1}\tAoo\btv(t)  + \tEoo^{-1}\btH\left(\btv \otimes \btv\right)  +   \tEoo^{-1}\tBo \bu(t). \label{eq:red_pod_divfree_Einv} 
	\end{eq} 
	Let us denote $\tilde\cD = \begin{bmatrix}  \tilde\bV \\ \tilde\bV\totimes\hat\bV \\ \bU \end{bmatrix}$, where $\tilde\bV = \begin{bmatrix}
		\btv_1, \btv_2, \dots, \btv_N
	\end{bmatrix}$ denotes the data matrix assembling $N$ snapshots of the projected reduced model \eqref{eq:red_pod_divfree_Einv} and \[\dot{\tilde\bV} = \begin{bmatrix}
		\frac{d}{dt}\btv(t_1), \frac{d}{dt}\btv(t_2), \dots, \frac{d}{dt}\btv(t_N)
	\end{bmatrix}\] 
	is its right-hand side matrix. Hence, the reduced operators $\tEoo^{-1}\tAoo$, $\tEoo^{-1}\btH$ and $\tEoo^{-1}\tBo $ represent one solution of the least-squares problem \eqref{eq:opinf_optimization_vel} for the data matrix $\tilde\cD$ and right-hand side matrix $\dot{\tilde\bV}$. Moreover, they represent the unique solution if the matrix $\tilde\cD$ has full rank. 
	
	Next, due to Assumption \ref{assum:scheme_conv}, notice that the projected velocities $\bhV$ are perturbations of the reduced velocities $\tilde\bV$, i.e., $\bhV = \tilde\bV + \delta \bV$, and, for $ \Delta t\rightarrow 0$ and the dimension $r$ increases, we have $\|\delta \bV\|_F \rightarrow 0$, because the reduced model \eqref{eq:red_pod_divfree_Einv} is equivalent to the full-order ODE \eqref{eq:nse-ode} in the POD basis of order $n_\bv$. Also, Assumption \ref{assum:deriv_conv} together with the limiting case $r = n_\bv$ leads to $\frac{d}{dt}\bhv(t_k) \rightarrow \frac{d}{dt}\btv(t_k)   $ as $r\rightarrow n_{\bv}$, for $k =1, \dots, N$. As a consequence, the data matrix $\cD$ and the approximated derivative information $\dbhV$ both can be interpreted as perturbation of $\tilde\cD$ and $\dot{\tilde\bV}$ respectively, i.e., $\cD = \tilde\cD + \delta \cD$ and $\dbhV =\dot{\tilde\bV} + \delta \dot{\bV}$, with $\| \delta \cD\|_F \rightarrow 0$ and $\|\delta \dot{\bV}\|_F \rightarrow 0 $  whenever $\Delta t \rightarrow 0$ and $r \rightarrow n_{\bv}$. Hence, as presented in \cite{morBenGKPW20}, this leads to the following asymptotic result for the least-squares problem:
	\[  \min_{\hat\bA_\bv,\hat\bH_\bv,\hat\bB_\bv}\left[ \lim\limits_{\substack{\Delta t \rightarrow 0 \\ r \rightarrow n_\bv}} \left\|\dot{\hat\bV} - \begin{bmatrix}\hat \bA_\bv, \hat\bH_\bv, \hat\bB_\bv \end{bmatrix} \cD \right\|_F\right] = 
	\min_{\hat\bA_\bv,\hat\bH_\bv,\hat\bB_\bv} \left\|\dot{\tilde\bV} - \begin{bmatrix}\hat \bA_\bv, \hat\bH_\bv, \hat\bB_\bv \end{bmatrix} \tilde\cD \right\|_F.\]
	As a consequence, the pre-asymptotic case combined with Assumption \ref{assum:fulrank_data} leads to the proof of the theorem.
\end{proof}

Theorem \ref{thm:POD_opinf} shows that the identified operators  converge to the POD projected operators as $r \rightarrow n_\bv$.  One particular limitation of this theorem is Assumption \ref{assum:alg_con}. Indeed, as mentioned in Remark \ref{rem:POD_algcond}, the SVD of the velocity snapshot matrix might generate POD modes that will not satisfy the algebraic constraints. This would result in intrusive velocity reduced-order models that are not fully independent of the pressure. On the other hand, the operator inferred model will still be a good approximation of the ODE velocity model \eqref{eq:nse-ode}.  Another limitation of this theorem is that the least-squares problem in \eqref{eq:opinf_optimization_vel} can be ill-conditioned. The numerical remedy for that is the regularizer procedure from Remark \ref{rem:l-curve}.  Although, in this case, Theorem \ref{thm:POD_opinf} does not hold, it will be shown in Section \ref{sec:num-results} that the operator inference can still identify reduced-order models with very accurate state errors.

\begin{remark}\label{rem:opinf_lin} 
	If our goal is to learn a linear system as folllows:
	\begin{equation}\label{eq:opinflin_red_vel}
		\dot{\hat\bv}(t) = \hat\bA_\bv\bv(t) + \hat\bB_\bv\bu(t),
	\end{equation} 
	that captures the dynamics of the nonlinear Navier-Stokes equations as close as possible, in the above described operator inference approach, we set $\hat\bH \equiv 0$. Such a linear model can be learned by solving 
	%
	%
	the above least-squares problem	
	\begin{equation}\label{eq:opinflin_optimization_vel}
		\min_{\hat\bA_\bv,\hat\bB_\bv} \left\|\dot{\hat\bV} - \begin{bmatrix}\hat \bA_\bv, \hat\bB_\bv \end{bmatrix} \cD \right\|_F, \quad \text{where $\cD = \begin{bmatrix}  \hat\bV  \\ \bU \end{bmatrix}$.}
	\end{equation}
\end{remark}

%% file: sec-inhomo-conti.tex
\section{Inhomogeneities in the Continuity Equation}\label{sec:ftwo-not-zero}
In case $\Aott \bv(t)-\bg(t)=0$ and $\bg \not\equiv 0$,  we decompose the state $\bv$ as $\bv(t) = \vz(t) + \vp(t)$, where $\vz(t) \in \ker \Aott$ that can be
defined as a solution to an ODE, and  $\vp(t)$ is an element of the complement of $\ker \Aott$.   This is typically done in the literature to compute projection-based reduced-order models in such cases, see, e.g., \cite{morHeiSS08,morAhmBGetal17}.  An advantage of this decomposition is that we can employ model reduction schemes developed for the homogeneous case, i.e., $\Aott \bv(t) = 0$. 
With the discrete
\emph{Leray} projection as in \eqref{eq:def-disc-leray}, one has
\begin{equation}\label{eqn:leray_decomp}
	\bv(t) = \Pi \bv(t) + (\bI-\Pi)\bv(t) =: \vz(t) + \vp(t)
\end{equation}
with the remainder part given as 
\begin{equation}\label{eq:eqn-vp}
	\vp(t) = - \Eoo^{-1}\Aot (\Aott \Eoo ^{-1}\Aot)^{-1}\ft(t).
\end{equation}
Assume that $\ft(t) = \bB_\perp \bu_\perp(t)$ with $\bu_\perp(t) \in \R$; thus, we can write $\vp(t) = \bS_\perp \bu_\perp(t)$.  In this case, the pressure $\bp$ is defined via
\begin{equation*}
	\bp(t) = -\bS^{-1}(\Aott \Eoo^{-1}\Aoo \bv(t) + \Aott \Eoo^{-1}\bH(\bv(t)\otimes \bv(t)) +
	\Aott \Eoo^{-1}\fo(t)
	- \bB_\perp \dot \bu_\perp(t) )
\end{equation*}
and the ODE \eqref{eq:nse-ode} can be formulated for $\vz$ via
\begin{equation}\label{eq:FOM-corrected}
	\begin{aligned}
		\Eoo \dvz(t) &= \Aoo \vz(t)+ \bH (\vz(t) \otimes \vz(t)) + \bN\vz(t)\bu_\perp(t) \\
		&\quad +  \bH \left(\bS_\perp \otimes \bS_\perp \right)\bu_\perp^2(t) +  \Aot \bp(t) + {\fo}(t),
	\end{aligned}
\end{equation} 
where $\bN = \bH \left( \bI \otimes \bS_\perp + \bS_\perp \otimes \bI \right) $. 

\begin{remark}
	Because $\Pi \dvp(t)\equiv 0$, one can show that $\dvp$ does not occur in
	\eqref{eq:FOM-corrected}; cp. \cite[Sec. 6]{morHeiSS08}.
\end{remark}

For a projection based reduction as described in Section
\ref{sec:projection-mor}, the following considerations are relevant:

\begin{itemize}
	\item If $\ft\equiv 0$ and the basis $\bV$ is divergence-free, i.e. $\Aott
	\bV=0$, then $\btA_{12} =0$ and the projected equation \eqref{eq:Red_DiffEq}
	recovers the corresponding projection of \eqref{eq:nse-ode}.
	
	\item If $\ft \neq 0$ and, thus, the corresponding basis $\bV$ is not
	divergence-free when they are computed using the original velocity snapshots, i.e., 
	\begin{equation}\label{eq:snap_matrix_non_divfree}
		\bV = \begin{bmatrix}
			\bv(t_0), \bv(t_1), \dots, \bv(t_N)
		\end{bmatrix},
	\end{equation}
	we first need to compute a divergence-free
	basis on the base of $\Pi \bV$ and $\vp(t)$ and apply the projection procedure to the corrected full order
	model \eqref{eq:FOM-corrected}.
	
\end{itemize}

Note that the resulting ODE equation has some additional terms if compared \eqref{eq:Red_DiffEq}; these are a bilinear term and square term of the input. But for a special case $\bu_\perp(t) \equiv \texttt{const}$, the bilinear term can be merged into the linear term, and the square input term becomes a constant. 
Furthermore, we mention that if the numerical realization of $\Pi$ is impossible or too costly, one may apply
\emph{empirically divergence-free} bases obtained by a projection defined by snapshots of the discrete pressure gradient
\begin{equation}
	\mathbb P := 
	\begin{bmatrix}
		\Aot\bp(t_0), \Aot\bp(t_1), \hdots, \Aot\bp(t_N)
	\end{bmatrix}.
\end{equation}
Let $\bQ_N$ be an orthogonal basis of $\sspan{ \mathbb P}$. Then, for the projection 
$\tilde \Pi:=\bI-\bQ_N \bQ_N^\top$, it holds that $\tilde \Pi \mathbb P = 0$. Consequently,
for a basis $\tVv=\tilde \Pi \tVv$, we have $\tVv^\top \Aot \bp(t_k)=0$ at all
snapshot locations $t_k$. Accordingly, for a Galerkin projection of
\eqref{eq:FOM-corrected} by $\tVv=\tilde \Pi \tVv$, one may well assume that
$\tAot=0$, since this condition holds at all data points.

When the exact projection $\Pi$ is known and $\bu_\perp  \equiv \texttt{const}$, the
approximation of the full state is defined as
\begin{equation}
	\bv(t) \approx \tVv^\top\hvz(t)+\vp(t) := \tVv^\top\hvz(t)+(\bI-\tilde \Pi)\bv(t_0) 
	\approx \tVv^\top\hvz(t)+\bQ_N\bQ_N^\top\bv(t_0),
\end{equation}
where the reduced state $\hvz$ solves the projected (and corrected) equations
\eqref{eq:FOM-corrected} with $\tAot$ set to zero; cp. also \eqref{eq:red_pod_divfree}.

The preceding considerations are only relevant for projection-based methods from Section~\ref{sec:projection-mor}.  These modifications are not necessary when a model is learned using data. It is based on the fact that the dynamics of velocity can be governed by an ODE regardless of whether the velocity satisfies divergence-free constraint or not. To elaborate more on it, we consider the algebraic equations as follows: 
\begin{equation}
	\Aott \bv(t) + \bB_\perp\bu_\perp(t) = 0,
\end{equation} 
implying 
\begin{equation}
	\Aott \dot\bv(t) + \bB_\perp\dot\bu_\perp(t) = 0.  
\end{equation} 
Next, we substitute the equation for $\dot\bv(t)$ from \eqref{eq:DisctNavSto}, we get 
\begin{equation}
	0 = \Aott \Eoo^{-1}\Aoo \bv(t) + \Aott \Eoo^{-1}\Aot \bp(t) + \Aott \Eoo^{-1}\bH
	(\bv(t) \otimes \bv(t)) + \Aott (\Eoo^{-1}\fo (t) + \bB_2\dot\bu_\perp(t)).
\end{equation}
With $\bS:=\Aott \Eoo^{-1} \Aot$ being invertible, the pressure $\bp$ can be
expressed in terms of $\bv$ and $\fo$ as follows:
\begin{equation}\label{eq:def-p_gen}
	\bp(t) = -\bS^{-1}\Aott \Eoo^{-1}(\Aoo \bv(t) + \bH(\bv(t)\otimes \bv(t)) + \fo(t) ) - \bS^{-1}\Aott\bB_2\dot\bu_\perp(t).
\end{equation}
Substituting $\bp(t)$ in \eqref{eq:DiffEq} from the above equation, we obtain an ODE for the velocity as follows:
\begin{equation}\label{eq:nse-ode_gen}
	\Eoo \dot \bv(t) = \Pi^\top\Aoo \bv(t) + \Pi^\top\bH(\bv(t)\otimes \bv(t)) +
	\Pi^\top\fo(t) - \Aot\bS^{-1}\Aott\bB_\perp \dot\bu_\perp(t) . 
\end{equation}
This illustrates that there exists an underlying ODE for the velocity, even in the inhomogeneous case. Hence, we can determine the underlying ODE by employing the proposed operator inference approach, discussed in \Cref{sec:opinf-nse}  with a slight modification. This is to include learning the operator for the derivative of the input $\bu_\perp$. 
This is one of the significant advantages of the non-intrusive method over the intrusive POD method.  

\begin{remark}
	One may argue that intrusive approaches can be applied to \eqref{eq:nse-ode_gen}, but this is highly undesirable since it requires the explicit computation of matrices such as $\Pi^\top \bA_{11}$ to determine projected reduced matrices.  Moreover, some matrices, e.g.,  $\Pi^\top \bA_{11}$  may lose sparsity, thus making intrusive model reduction computationally very expensive.
\end{remark} 
%


%% file: sec-numerical-examples.tex
\section{Numerical experiments}\label{sec:num-results}
In this section, we illustrate the efficiency of the proposed method using numerical experiments and compare with the intrusive POD method and the data-driven method DMD. Precisely,  the following methodologies are compared:
\begin{itemize}
	\item  \opinf:  the proposed operator inference from Section \ref{sec:opinf-nse},
	\item \opinflin: its linear variant described in Remark \ref{rem:opinf_lin}, 
	\item \POD: intrusive model reduction based on POD, see Section \ref{sec:projection-mor}, 
	\item \DMD: classical DMD approach, see \cite[Algorithm 1]{TuRLBK14},  
	\item \DMDc: DMD with control, see \cite[Algorithm 3.4]{ProBK16}), and
	\item \DMDquad: DMD with quadratic term and control, see \cite[Section
    3.2]{morGosP20}.
\end{itemize}
Moreover, below, we provide some technical information that applies to the construction of the surrogate models and the integration of the  reduced systems: 
\begin{itemize}
	\item  \POD, \opinf~and \opinflin~surrogate models are simulated using the Python routine \texttt{odeint} from \texttt{scipy.integrate}.
	\item To employ the operator inference approach, we require the time-derivative approximation of $\bhv(t)$. This is obtained using the reduced trajectory $\hat\bv(t)$ by employing a fourth-order Runge-Kutta scheme. 
	\item The  DMD procedures are applied to the projected velocity trajectories \eqref{eq:red_vel} and then lifted using the POD basis. 
	\item We assume that the \POD~surrogate models are pure ODE models as
    \eqref{eq:red_pod_divfree}, i.e., $\btA_{12} = 0$ in
    \eqref{eq:pod_red_matrices}. For that we apply the necessary corrections as
    indicated by the theory laid out in Section \ref{sec:ftwo-not-zero}.
    However, since this is not in the scope of the standard POD approach, we
    will not take particular measures against the error introduced by
    inaccuracies in the SVD (cp. Remark \ref{rem:POD_algcond}).
\end{itemize}

The code and the raw data is made available as mentioned in Figure
\ref{fig:linkcodndat}.

\subsection{A low-order demo example}\label{subsec:demoexample}
In order to illustrate the proposed approach, we create a low-order random demo example of the form \eqref{eq:DisctNavSto}. We have generated it using the seed $0$ with $n_\bv = 4$ and $n_\bp = 1$. We construct the data for velocity and pressure using an input $u(t) = \sin(2t)e^{-0.05t}$ with zero initial condition. Next, we learn surrogate models using \opinf, \DMD, \DMDc, and \DMDquad. We set the tolerance $\tol = 10^{-4}$ for the SVD (step 5 of Algorithm \ref{algo:opinf}). In \Cref{fig:random_comparison_ori_opInf}, we compare the time-domain simulations of the original and inferred model using \opinf. Moreover, we plot the mean-absolute error of the velocity of the original and various inferred models in \Cref{fig:random_error}. We observe that \opinf~learns the best model among the considered approaches that capture the dynamics of the given data. DMD approaches that aim to learn a linear model fail to capture the dynamics; however, DMD with quadratic observers improves the model as one can expect. 

\begin{figure}[!tb]
	\centering
	\setlength\figureheight{6.0cm}
	\setlength\figurewidth{5cm}
	\input{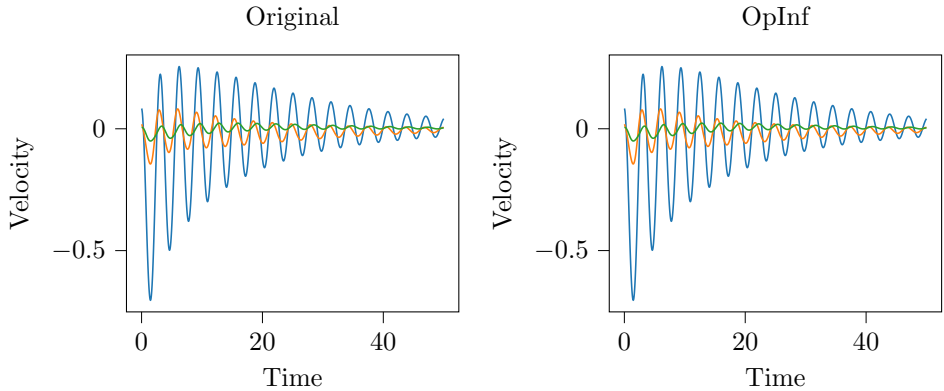}
	\caption{A comparison of the velocity of the original model and inferred model using \opinf.}
	\label{fig:random_comparison_ori_opInf}
\end{figure}

\begin{figure}[!tb]
	\centering
	\setlength\figureheight{6.0cm}
	\setlength\figurewidth{10cm}
	\includegraphics[width =\figurewidth, height =\figureheight]{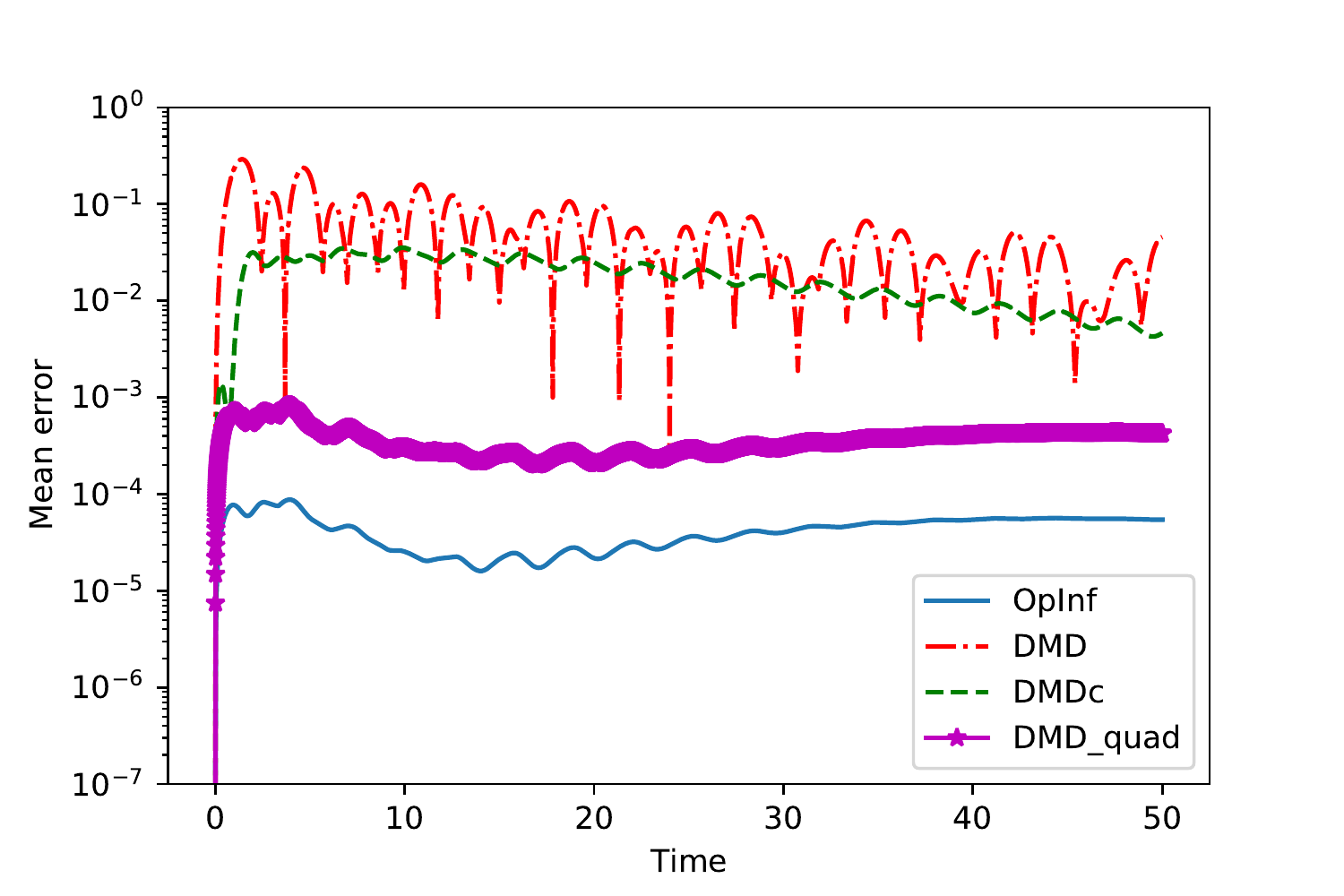}
	\caption{A comparison of the velocity obtained from the original model and inferred models using various approaches.}
	\label{fig:random_error}
\end{figure}

\subsection{Navier-Stokes examples}
We consider two test cases modeled by incompressible Navier-Stokes equations; namely,  the \emph{lid driven cavity} and  \emph{cylinder wake}, to
illustrate the performance and highlight certain properties of the non-intrusive
 schemes.  Both examples consider two-dimensional flows. For the spatial discretization, we use \emph{Taylor-Hood} elements with piece-wise quadratic ansatz  functions for the
velocities and piece-wise linear ansatz functions for the pressure on a 
triangulation; see Figure \ref{fig:num-setups} for an illustration of the setup. 

As described, e.g.,  in \cite{BehBH17}, a spatial discretization of the problems leads
to a system of the form (cp. \eqref{eq:DisctNavSto})
\begin{subequations}\label{eq:numex-semidisc-nse}
	\begin{align}
		\Eoo\dot{\bv}(t) &= \Aoo\bv(t) + \Aot\bp(t) + \bH\left(\bv(t) \otimes \bv(t)\right)  +  \mathbf{f}(t), \\ 
		0 &=  \Aott\bv(t) + \bg(t),\\
		y_{\bv}(t) &= \bC_{\bv} \bv(t), \\
		y_{\bp}(t) &= \bC_{\bp}\bp(t),
	\end{align}
\end{subequations}
where $y_{\bv}$ and $y_{\bp}$  denotes  characteristic outputs that is derived from the state $\bv$ and $\bp$
via a linear map, respectively; cp. \cite[Ch. 8]{BehBH17}.

\begin{figure}[tb]
	\centering
	\setlength\figureheight{5cm}
	\setlength\figurewidth{5cm}
	\input{figures/drivcav_grid.tikz}
	
	\setlength\figureheight{4cm}
	\setlength\figurewidth{15cm}
	\input{figures/cylwake_grid.tikz}
	\caption{An illustration of the computational domains of the \emph{driven cavity
			flow} (above) and and \emph{flow around a cylinder} (below), and example discretizations of the respective domains.} 
	\label{fig:num-setups}
\end{figure}
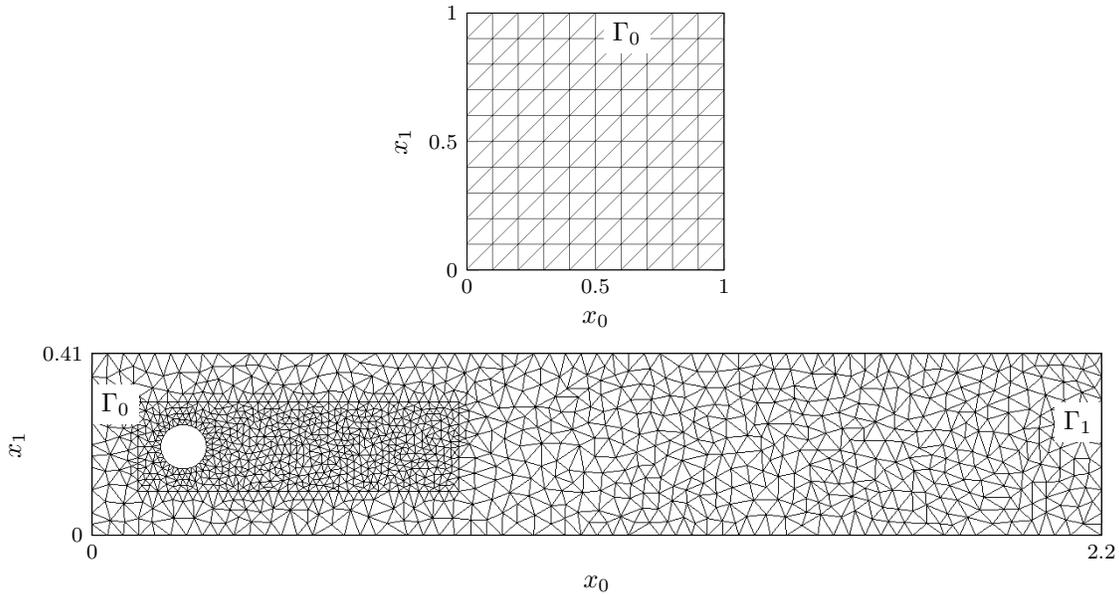

The \DC~example describes an internal flow in a cavity that is driven by a
moving lid (depicted by $\Gamma_0$ in Figure \ref{fig:num-setups}) modeled by
a Dirichlet boundary condition for the tangential component of the velocity. The
resulting semi-discrete system is as in \eqref{eq:numex-semidisc-nse} with
$\ft(t)\equiv 0$. On the other hand, the \CW~example models flow through a channel with a round obstacle in
between. The flow is induced by a strongly imposed parabolic flow profile at
the inlet (depicted by $\Gamma_0$ in Figure \ref{fig:num-setups}). The
corresponding semi-discrete model is in the form of \eqref{eq:numex-semidisc-nse} with
$\ft(t)\equiv \mathop{const}$.

\subsubsection{Numerical setup and computation of the snapshots}

As for the discretization of the time dimension, we apply an
\emph{implicit-explicit} Euler scheme -- i.e., an explicit treatment of the nonlinear term and implicit treatment of the linear term and, in particular, the
constraint equation in \eqref{eq:numex-semidisc-nse} -- on an equidistant grid.
For the initial value, we use the solution of the corresponding steady-state
\emph{Stokes problem}. 

To probe the accuracy of the discretization and, thus, the reliability of the
snapshots and the resulting reduced-order models, we compute the resulting
discrete approximation of the velocities $v$
\begin{equation}
	\int_{0}^{T} \|v(t)\|_{L^2(\Omega)} dt,
\end{equation}
where $T$ is the end time (see Table \ref{tab:num-setups}) and $\Omega$ is the domain of the problems,
and compare it to the corresponding approximation based on a finer spatial grid
and a high-accuracy time integration scheme with step-size control applied to
the ODE formulation \eqref{eq:nse-ode}. As indicated by the numbers presented in
Table \ref{tab:conv-tests}, the error in the spatial discretization is
dominating, so that the linear convergence of the time integration scheme only
occurs on the finest triangulation.

\begin{table}
	\centering
	\begin{tabular}{l|c|c}
		Parameter & \DC & \CW \\
		\hline
		Reynolds number & 500 & 60 \\
    state dimensions $(n_\bv, n_\bp)$ & (3042,441) & (9356,1289) \\
		time interval & [0, 6] & [0, 2] \\
		number of timesteps & 512 & 512 \\
		number of snapshots & 513 & 513 \\
		order of discretized model \eqref{eq:numex-semidisc-nse} & 3042 &  5812
	\end{tabular}
	\caption{The parameters of the numerical setups.}\label{tab:num-setups}
\end{table}

\begin{table}
	\centering
	\caption*{\DC}
	\begin{tabular}{l|ccc}
		$n_\bv$/ number of timesteps & $256$ & $512$ & $1024$ \\
		\hline
		$722$ & $-0.0988$ & $-0.0992$ & $-0.0993$ \\
		$3042$ & $-0.0292$ & $-0.0299$ & $-0.0303$ \\
		$6962$ & $0.0015$ & $0.0007$ & $0.0003$
	\end{tabular}
	\caption*{\CW}
	\begin{tabular}{l|ccc}
		$n_\bv$/ number of timesteps & $512$ & $1024$ & $2048$ \\
		\hline
		$5812$  & $-0.0022$ & $-0.0025$ & $-0.0026$ \\
		$9356$  & $-0.0002$ & $-0.0005$ & $-0.0006$ \\
		$19468$ & $0.0005$ & $0.0002$ & $0.0001$
	\end{tabular}
	\caption{The difference in the numerical approximation of the norm of the
		computed velocity to the approximation on the finest spatial discretization with a
		high-accuracy time integration scheme.}
	\label{tab:conv-tests}
\end{table}


\begin{remark} For the large-scale Navier-Stokes examples,  \DMDquad~produced only unstable models; therefore, the results will not be reported.
\end{remark}


\subsubsection{Example 1: driven cavity}
Here, we report the numerical experiments conducted for the \DC~example. We consider the full-order model (\FOM) semi-discretized in space  of order $n_\bv = 3042$.  To infer models, $513$ velocity snapshots are generated in the interval $[0,6]$ equally spaced in time using the \emph{implicit-explicit} Euler scheme. It is again worth noticing that, for this example, $\ft(t)\equiv 0$. As a consequence, the velocity snapshot matrix $\bV$ satisfies the algebraic constraints, i.e., $\Aott \bV = 0$.  Then, the POD basis $\bV_\bv$ was constructed using the SVD of the snapshot matrix $\bV$. Figure \ref{fig:dc_decay_sv} depicts the decay of the singular values of the snapshot matrix and shows how much the algebraic conditions are satisfied if the POD basis is increased, i.e., \[ \dfrac{\|\Aott\bV_\bv\|_F}{\|\bV_\bv\|_F},\] as function of the reduced order $r_\bv$. As expected, the POD basis does not satisfy the algebraic constraints if a larger order for the reduced model is chosen. Consequently, for high orders of the reduced models, the POD bases are not entirely divergence-free, leading to surrogate models that have a non-neglectable influence of the pressure. 
Ideally, these POD bases should be divergence-free, but in standard
implementations in Python and MATLAB, singular vectors corresponding to smaller
singular values are computed with low accuracy; cp. Remark \ref{rem:POD_algcond}. 

\begin{figure}[tb]
	\centering
	\setlength\figureheight{6.5cm}
	\setlength\figurewidth{10cm}
	\includegraphics[width =\figurewidth, height =\figureheight]{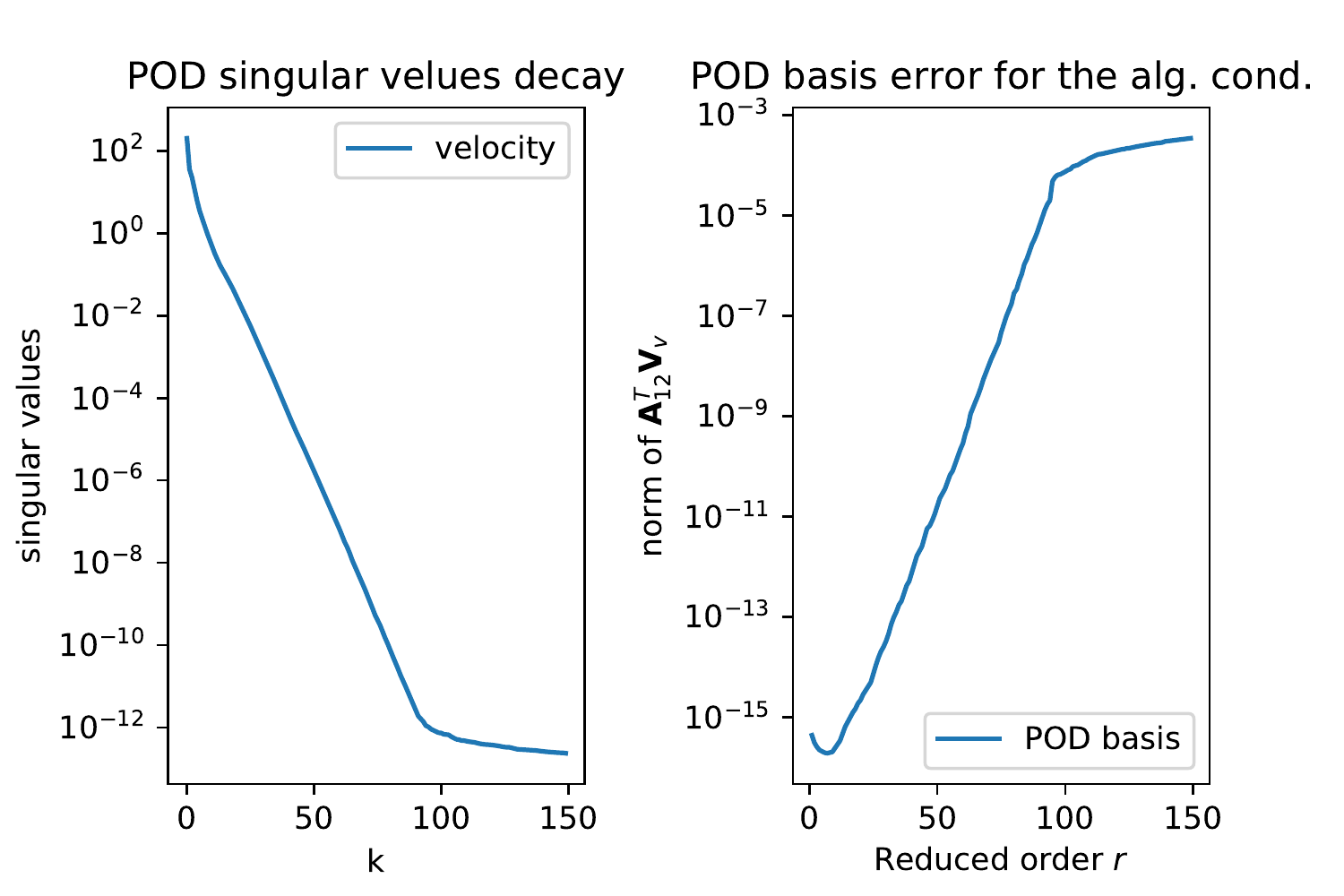}
	\caption{\DC: Decay of singular values and POD basis error for algebraic conditions.}
	\label{fig:dc_decay_sv}
\end{figure}

We infer reduced-order models of order $r_\bv = 30$ using \opinf, \opinflin, \POD, \DMD,~and \DMDc.  
As mentioned in Remark \ref{rem:l-curve}, the least-squares problem arising in \opinf~can be ill-conditioned; thus, the choice of $\texttt{tol}$ in step 5 of Algorithm~\ref{algo:opinf} plays a crucial role. Hence, for different tolerances in the range of $[10^{-11}, 10^{-6}]$, Figure \ref{fig:dc_l_curve} shows the L-curve, i.e., for each given $\texttt{tol}$, a scatter plot between the least-squares error and the norm of the least-squares solution. Using the L-curve criterion, we set the tolerance as $\texttt{tol} = 10^{-7}$. \Cref{fig:dc_time_domain_analysis} depicts the time-domain simulation of some observed trajectories for order $r_\bv = 30$.  It also shows the error $\|\bv_{\FOM}(t) - \bhv_{\ROM}(t)\|_F^2$ committed by the different surrogate models. Surprisingly, although the \FOM~had an intrinsic nonlinear behavior,  \opinf, \opinflin~and \DMDc~have a similar performance on this model. These three methodologies outperform \DMD~and \POD.  Additionally, by solving the least-squares problem \eqref{eq:opinf_optimization_pre}, one obtains the \opinf~reduced model for the pressure. The time-domain evolution of the pressure for the \FOM~and \opinf~models are depicted in Figure \ref{fig:dc_pressure} as well as the approximation error.   

\begin{figure}[tb]
	\centering
	\setlength\figureheight{6.5cm}
	\setlength\figurewidth{10cm}
	\includegraphics[width =\figurewidth, height =\figureheight]{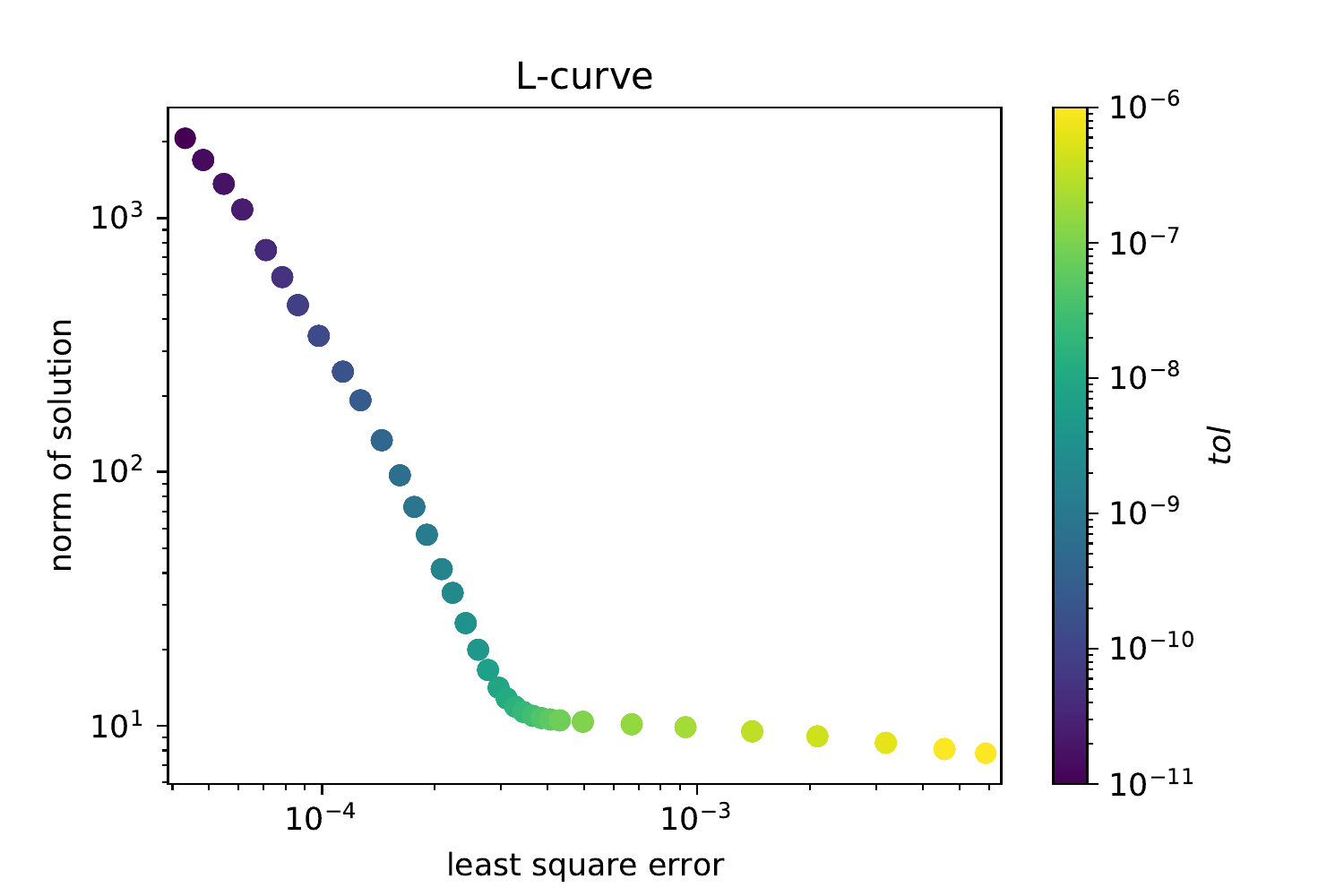}
	\caption{\DC: L-curve for $r_\bv = 30$}
	\label{fig:dc_l_curve}
\end{figure}

\begin{figure}[tb]
	\centering
	\setlength\figureheight{8cm}
	\setlength\figurewidth{12cm}
	\includegraphics[width =\figurewidth, height =\figureheight]{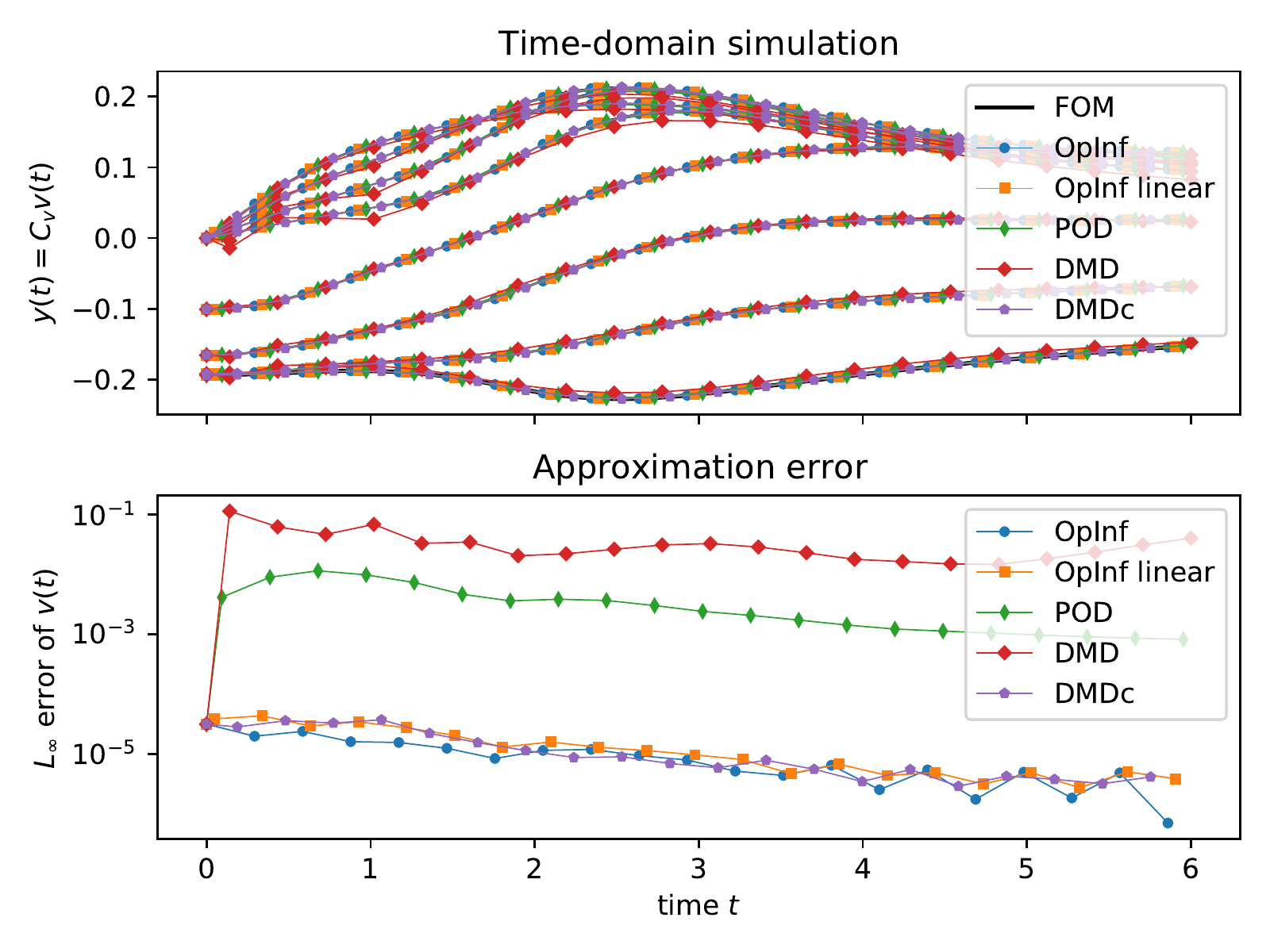}
	\caption{\DC: A comparison of time-domain simulations of the original and reduced-order models for $r_\bv= 30$.}
	\label{fig:dc_time_domain_analysis}
\end{figure}

\begin{figure}[tb]
	\centering
	\setlength\figureheight{8cm}
	\setlength\figurewidth{12cm}
	\includegraphics[width =\figurewidth, height =\figureheight]{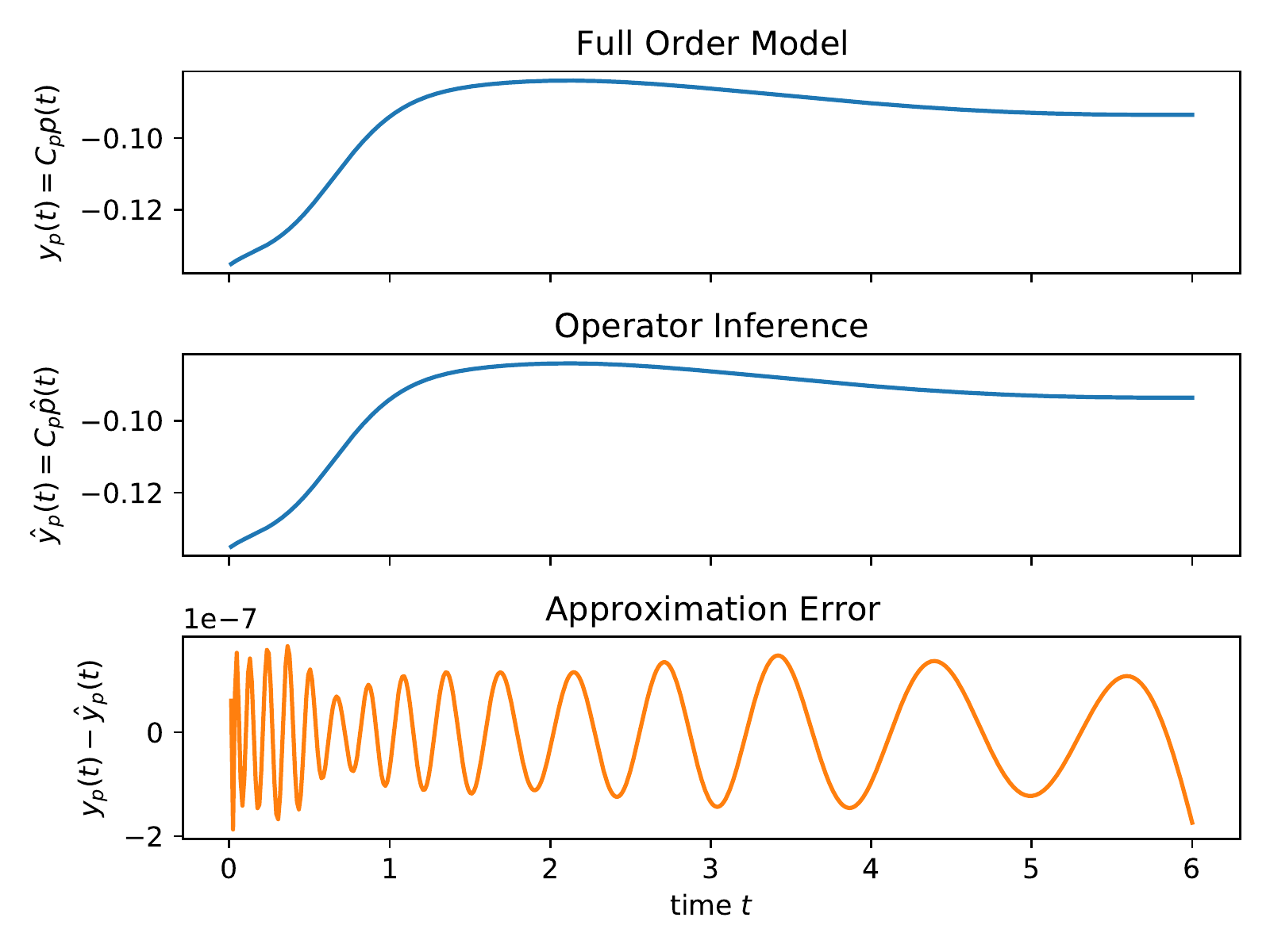}
	\caption{\DC: Time-domain evolution of the observed pressure (top), the values
  for the inferred model for order $r_\bv= 30$ and the difference in the
outputs.}
	\label{fig:dc_pressure}
\end{figure}

Next, we compute the relative time-domain $L_2$ errors for the order of the reduced-order models, ranging from  $6$ to $30$. Figure  \ref{fig:dc_error_decay_vs_order} depicts the decay of those errors for the different methodologies. For this example, the \POD~and \DMD~methodologies reach its performance limitations after some given reduced order. Once again,   \opinf, \opinflin,~and \DMDc~have a similar performance for the different reduced orders. This shows that, for this example, the presence of a control matrix $\bhB$  in the surrogate model structure is crucial for learning the dynamics of the \FOM, and the quadratic structure  $\bhH$ seems not to play an important role.

\begin{figure}[tb]
	\centering
	\setlength\figureheight{6.5cm}
	\setlength\figurewidth{10cm}
	\includegraphics[width =\figurewidth, height =\figureheight]{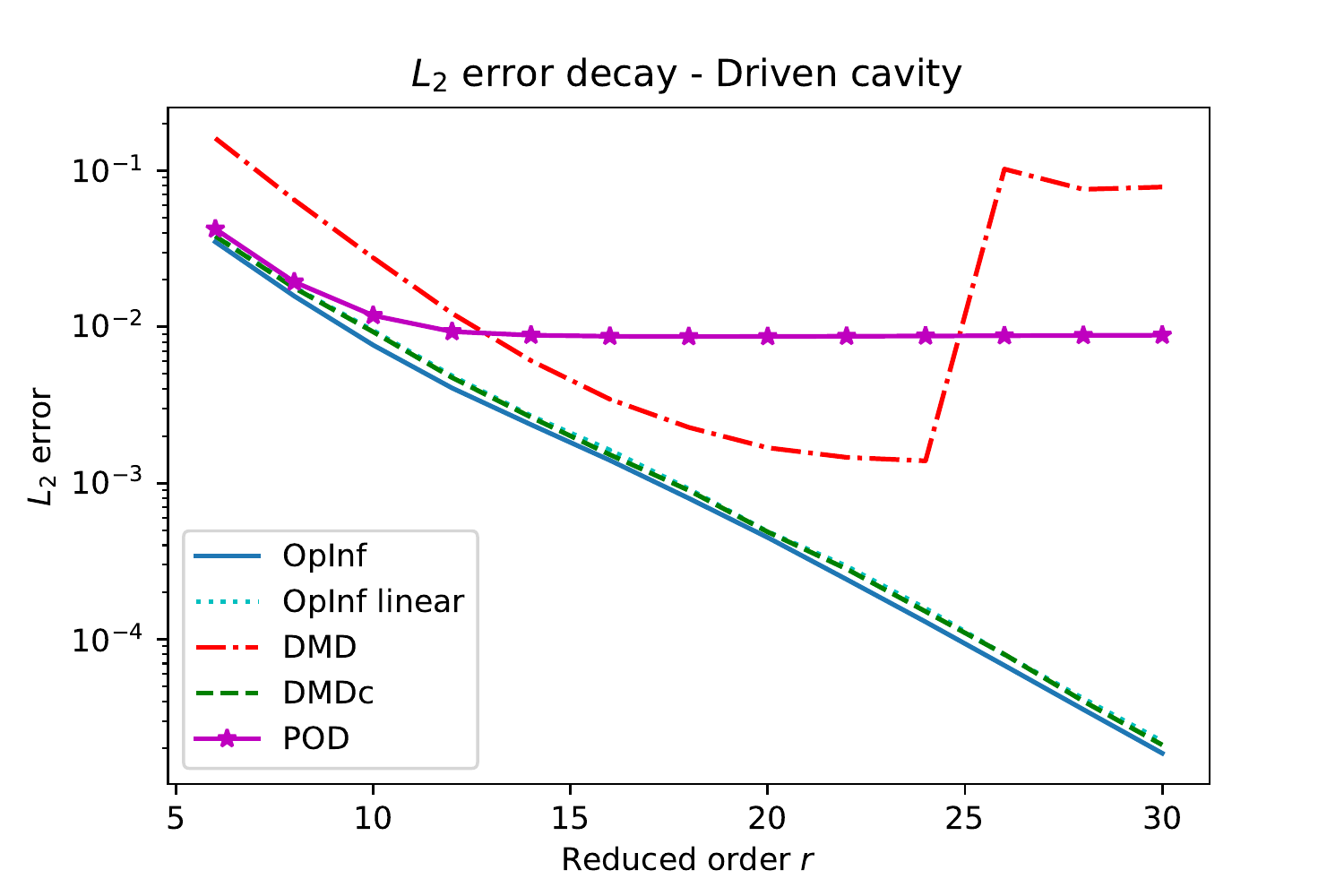}
	\caption{\DC: $L_2$ errors between the velocities obtained using the original and other reduced-order models,  obtained  using various method and orders.}
	\label{fig:dc_error_decay_vs_order}
\end{figure}

\subsubsection{Example 2: cylinder wake}
Now, we report the numerical experiments conducted for the \CW~example. We have considered the semi-discretized in space \FOM~of order $n_\bv = 5812$.  For this example, we again compute $513$ velocity snapshots in the interval $[0,2]$ equally spaced in time which are generated using the \emph{implicit-explicit} Euler scheme. Note that $\ft(t)\equiv \texttt{const}$ and, hence, the snapshot matrix $\bV$ does not satisfy the algebraic constraints. However, as explained in Section \ref{sec:ftwo-not-zero}, we can construct the divergence-free snapshot matrix 
\[\bV_0 = \Pi \bV = 
\begin{bmatrix}
	\vz(t_0), \vz(t_1), \dots, \vz(t_N)
\end{bmatrix}
\] by means of the Leray projection, see \eqref{eqn:leray_decomp}, where
$\vz(t)$ is the divergence-free velocity and, hence, $\Aott \bV_0 = 0$ holds.  The POD surrogate model will then be computed using the velocity free-snapshots as stated in Section \ref{sec:ftwo-not-zero}.
On the other hand, data-driven methods such as \opinf~and \DMD~to infer reduced models do not require such a transformation, and the original velocity snapshots can directly be used to learn models. 
This is due to the fact that there is a hidden quadratic ODE for the velocity $\bv(t)$. Next, two POD bases are computed using the snapshot matrix $\bV$ and another for the divergence-free velocity $\bV_0$.  Figure \ref{fig:cw_pod_decay_and_alg_cond}  depicts the decay of the singular values of these snapshot matrices as well as the algebraic error committed by the divergence-free POD basis. 
\begin{figure}[tb]
	\centering
	\setlength\figureheight{6.5cm}
	\setlength\figurewidth{10cm}
	\includegraphics[width =\figurewidth, height =\figureheight]{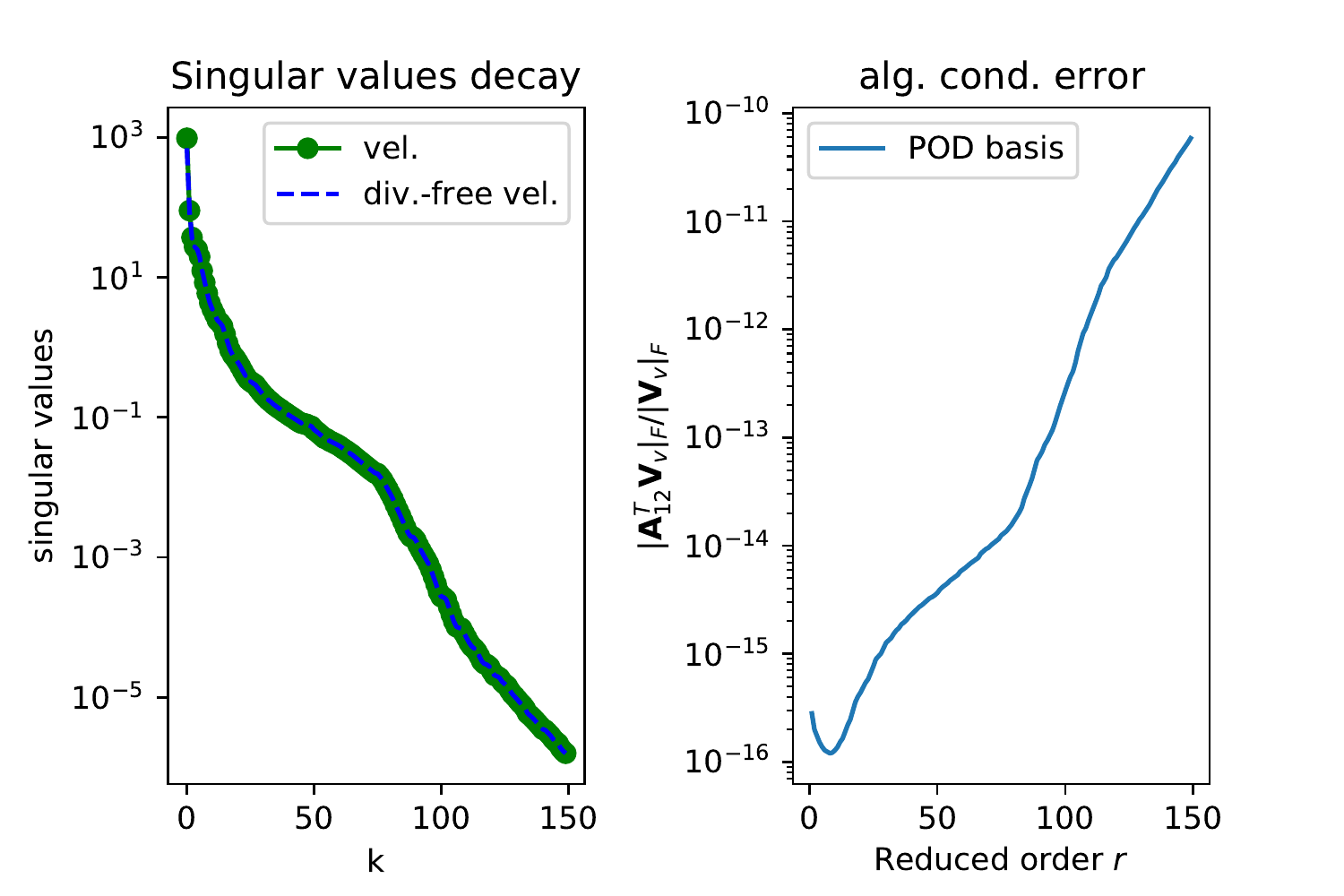}
	\caption{\CW: We plot the decay of singular values (left) and show how the divergence-free POD basis satisfies the  algebraic constraints (right).}
	\label{fig:cw_pod_decay_and_alg_cond}
\end{figure}

Next, for $r_\bv = 30$, surrogate models are obtained using the different methodologies. To compare the qualities of these models, we perform time-domain simulations and compare them in \Cref{fig:cw_time_domain_analysis}. We did not include the time-domain simulation for \DMD, since the results could not follow the \FOM~behavior at all.  Additionally, Figure \ref{fig:cw_time_pressure} shows the pressure evolution obtained using \opinf. Moreover,  we compute the relative time-domain $L_2$ errors for different orders between  $8$ and  $32$,  and the results are depicted in Figure  \ref{fig:dc_error_decay_vs_order}. As in the \DC~example, \opinf, \opinflin~and \DMDc~have a similar performance for the different reduced orders. Hence, for this example,  the control term $\bhB$ is crucial for the identification of \FOM.

\begin{figure}[tb]
	\centering
	\setlength\figureheight{8cm}
	\setlength\figurewidth{12cm}
	\includegraphics[width =\figurewidth, height =\figureheight]{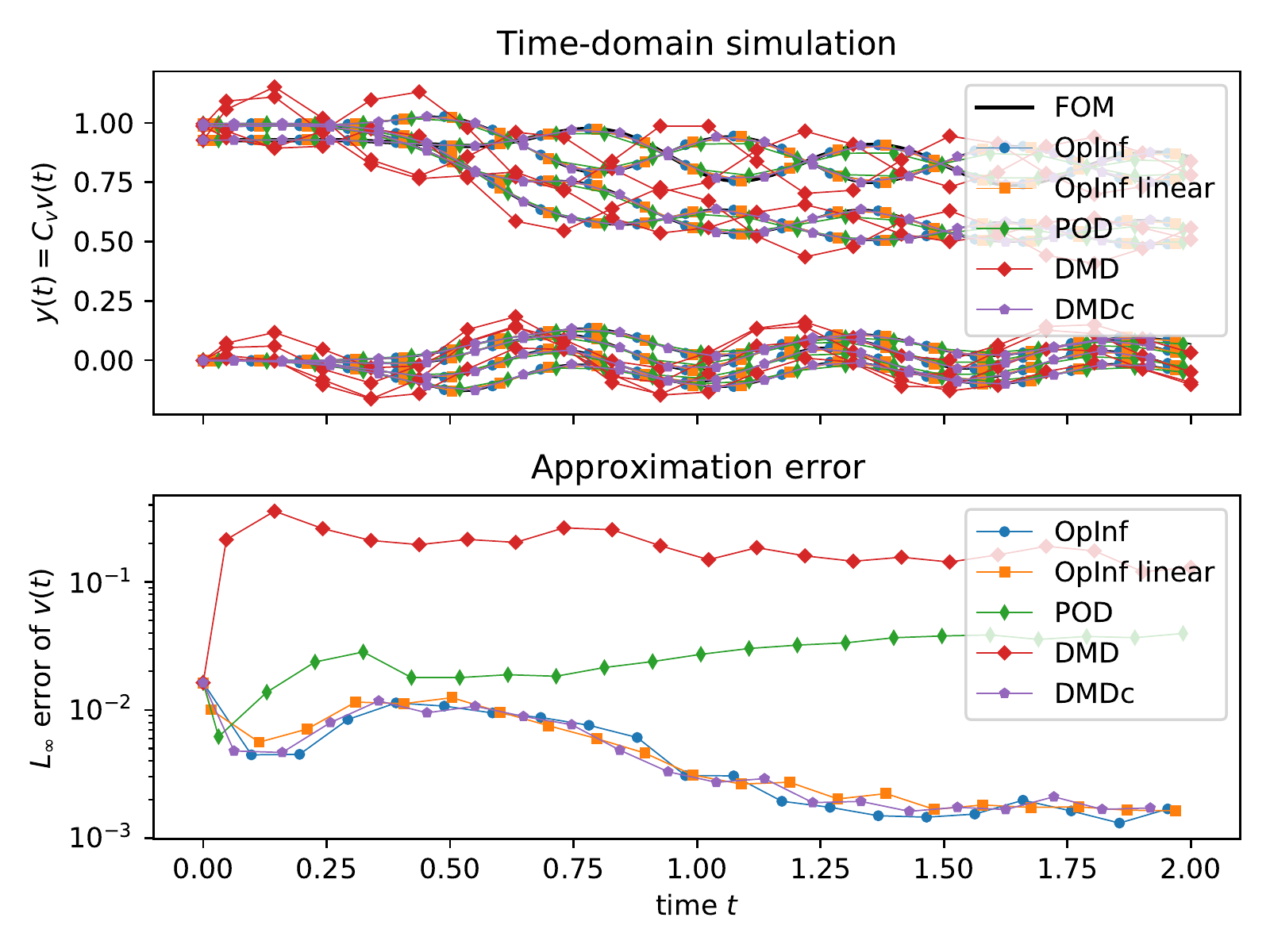}
	\caption{\CW: A comparison of time-domain simulation of different models.}
	\label{fig:cw_time_domain_analysis}
\end{figure}
\begin{figure}[tb]
	\centering
	\setlength\figureheight{8cm}
	\setlength\figurewidth{12cm}
	\includegraphics[width =\figurewidth, height =\figureheight]{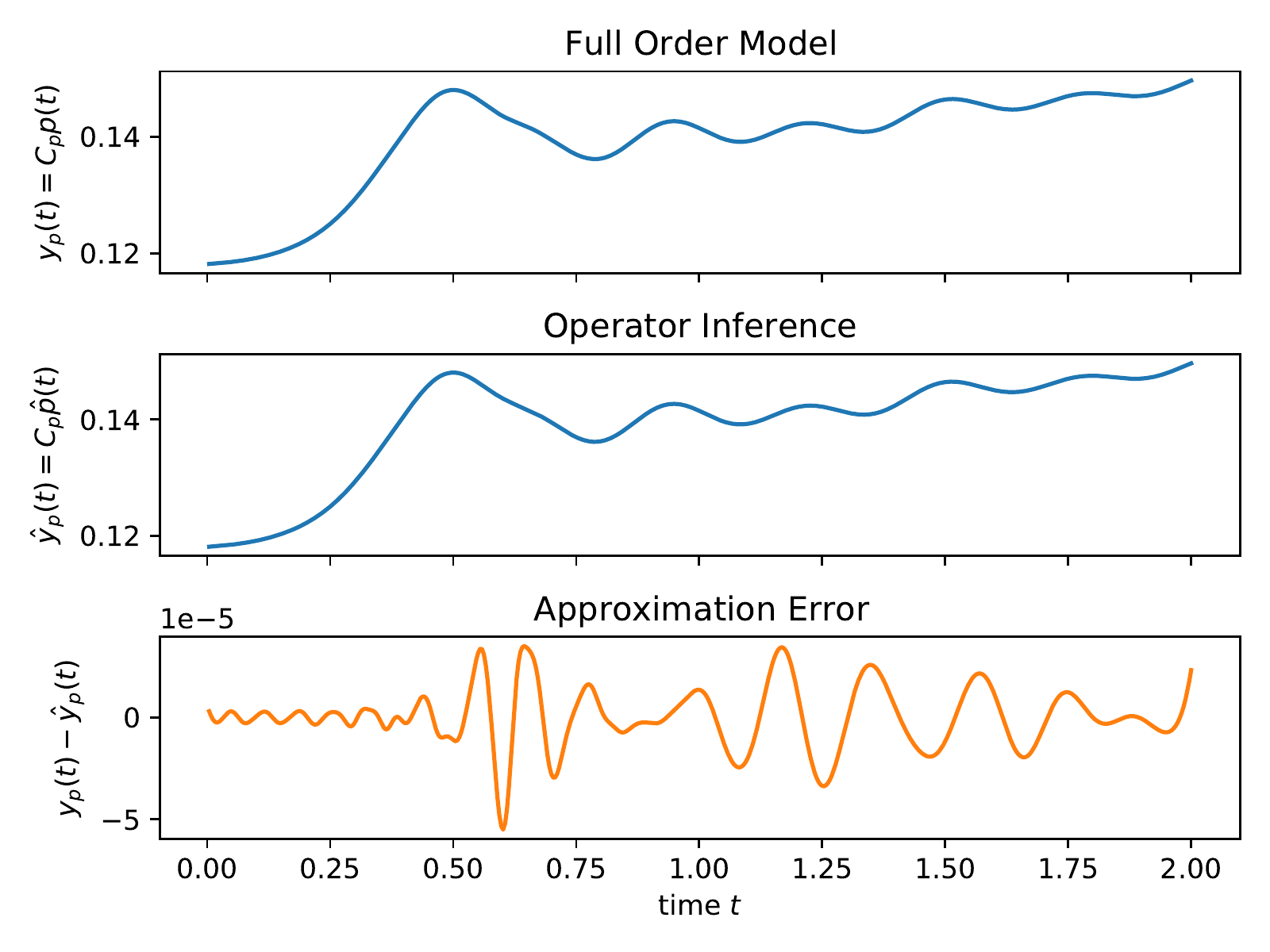}
	\caption{\CW: Comparison of the pressure output for the original model and the inferred model using \opinf.}
	\label{fig:cw_time_pressure}
\end{figure}

\begin{figure}[tb]
	\centering
	\setlength\figureheight{6.5cm}
	\setlength\figurewidth{10cm}
	\includegraphics[width =\figurewidth, height =\figureheight]{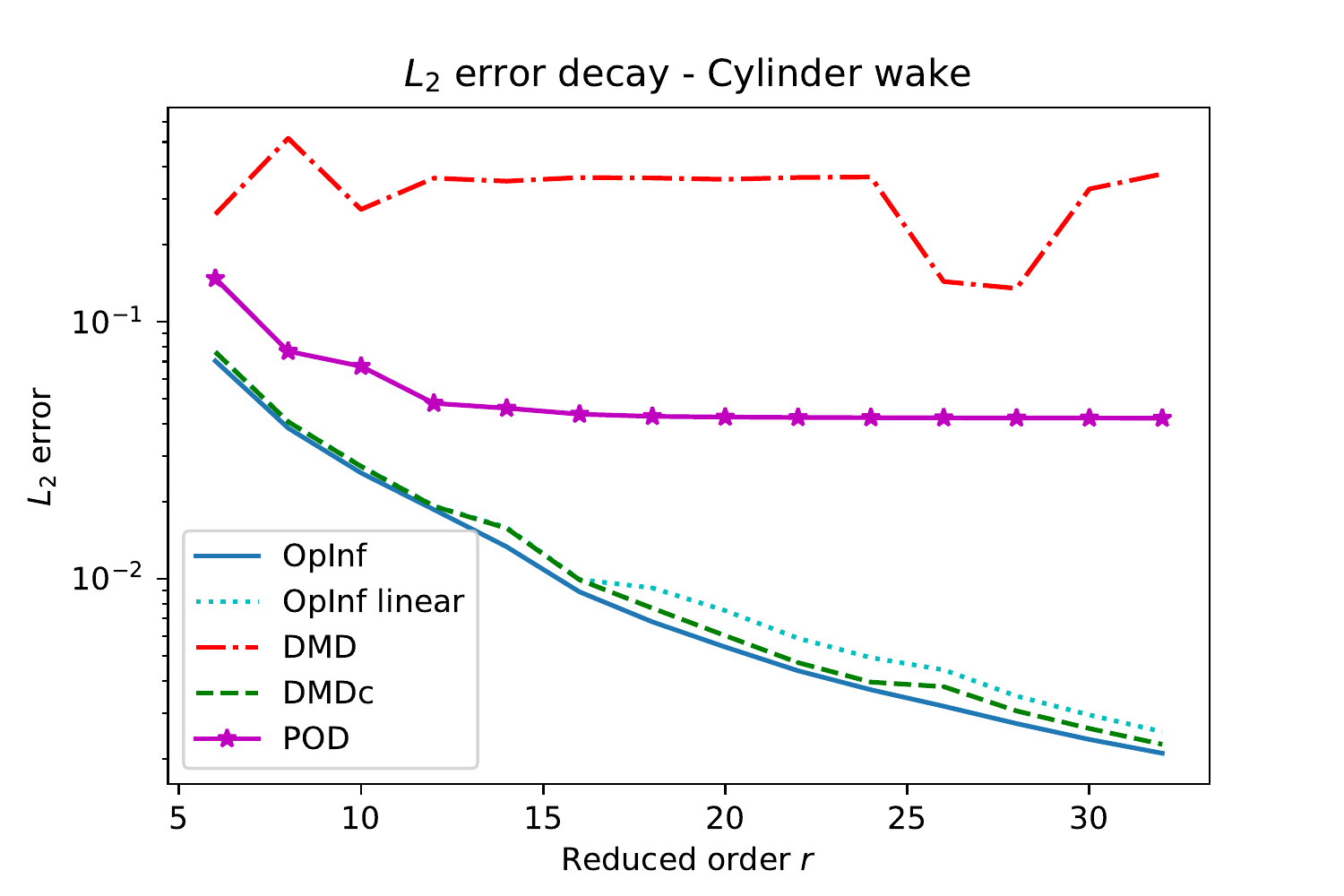}
	\caption{\CW: $L_2$ errors between the velocities obtained using the original and other reduced-order models,  obtained  using various method and orders.}
	\label{fig:cw_error_decay_vs_order}
\end{figure}

\begin{figure}
	\begin{minipage}{.7\textwidth}
		\centering
		\includegraphics[width=\textwidth]{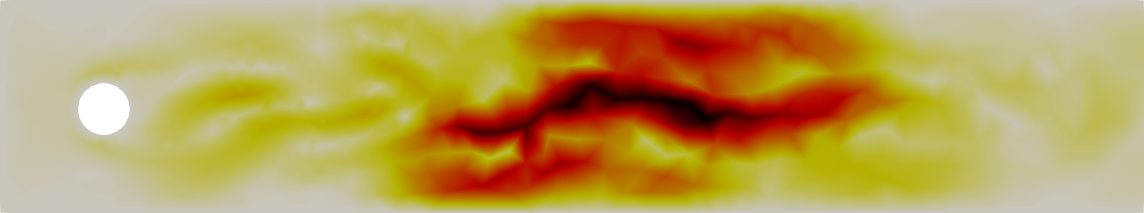}
		\includegraphics[width=\textwidth]{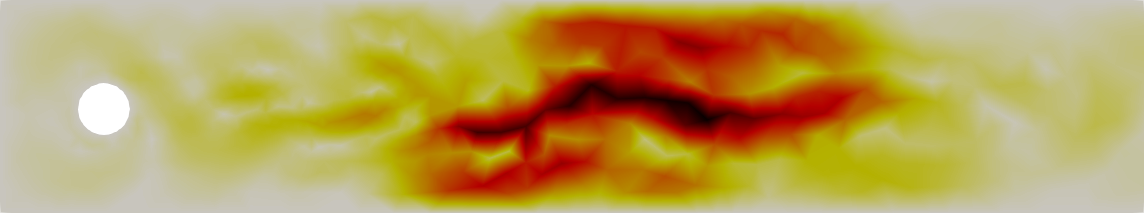}
	\end{minipage}%
	\begin{minipage}{.3\textwidth}
		\centering
    \includegraphics[width=.7\textwidth]{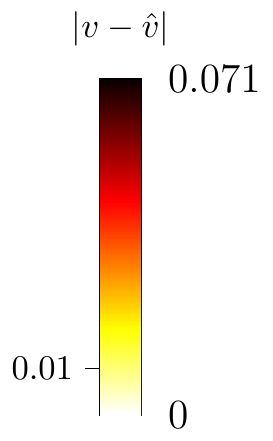}
	\end{minipage}
	\caption{Errors $|\bv(t) - \hat \bv(t)|$ for $r_\bv=30$ at $t=2$ for the approximation
		via \DMDc~(top) and \opinf~(bottom).}
\end{figure}

\begin{figure}[h]
	\begin{framed}
		\textbf{Code and Data Availability} \\
		The source code of the implementations, the raw data, and the scripts
		for computing the presented results are available via 
\begin{center}
  \href{https://doi.org/10.5281/zenodo.4086018}{\texttt{doi:10.5281/zenodo.4086018}}
\end{center}
		and GitHub under the MIT license.
	\end{framed}
		\caption{Link to code and data.}\label{fig:linkcodndat}
\end{figure}

%% file: figures/drivcav_grid.tikz
%
%
%
%
\begin{tikzpicture}

\begin{axis}[
xmin=0, xmax=1,
ymin=0, ymax=1,
width=\figureheight,
height=\figurewidth,
axis on top,
xtick={0,.5,1},
ytick={0,.5,1},
tick label style={font=\footnotesize},
xlabel={$x_0$},
ylabel={$x_1$},
xlabel near ticks,
ylabel near ticks
]
\addplot graphics [includegraphics cmd=\pgfimage,xmin=0, xmax=1, ymin=0, ymax=1] {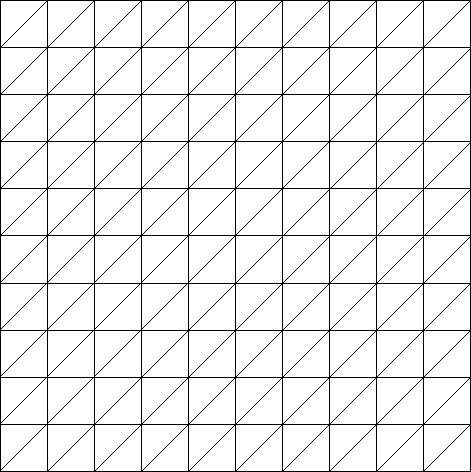};

\path [draw=black, fill opacity=0] (axis cs:13,1)--(axis cs:13,1);

\path [draw=black, fill opacity=0] (axis cs:0.05,13)--(axis cs:0.05,13);

\path [draw=black, fill opacity=0] (axis cs:13,1.38777878078145e-17)--(axis cs:13,1.38777878078145e-17);

\path [draw=black, fill opacity=0] (axis cs:0,13)--(axis cs:0,13);

\node at (axis cs:0.53,0.84) [color=black, fill=white, anchor=south west] {$\Gamma_0$};

\end{axis}

\end{tikzpicture}

%% file: figures/cylwake_grid.tikz
%
%
%
%
\begin{tikzpicture}

\begin{axis}[
xmin=0, xmax=2.2,
ymin=0, ymax=.41,
width=\figurewidth,
height=\figureheight,
axis on top,
xtick={0,2.2},
ytick={0,.41},
tick label style={font=\footnotesize},
xlabel={$x_0$},
ylabel={$x_1$},
xlabel near ticks,
ylabel near ticks
]
\addplot graphics [includegraphics cmd=\pgfimage,xmin=0, xmax=2.2, ymin=0, ymax=0.41] {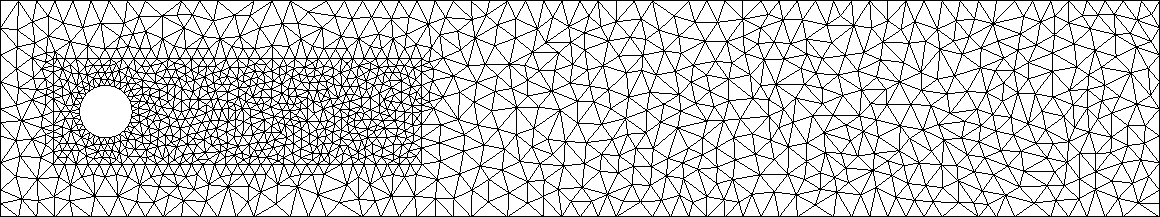};

\path [draw=black, fill opacity=0] (axis cs:13,1)--(axis cs:13,1);

\path [draw=black, fill opacity=0] (axis cs:0.05,13)--(axis cs:0.05,13);

\path [draw=black, fill opacity=0] (axis cs:13,1.38777878078145e-17)--(axis cs:13,1.38777878078145e-17);

\path [draw=black, fill opacity=0] (axis cs:0,13)--(axis cs:0,13);

\node at (axis cs:0.,0.25) [color=black, fill=white, anchor=south west] {$\Gamma_0$};
\node at (axis cs:2.2,0.21) [color=black, fill=white, anchor=south east] {$\Gamma_1$};

\end{axis}

\end{tikzpicture}

%% file: sec-conclusion.tex
\section{Conclusions}
In this paper, we have  tailored the operator inference approach \cite{morPehW16} to
learn physics-informed low-order models for incompressible flows using data.
We have shown that although the velocity and pressure are coupled in incompressible flow problems, one can identify a differential equation using only velocity snapshots that determine the velocity field's evolution. Moreover, there exists an algebraic equation linking pressure and velocity. In Koopman's philosophy, to learn dynamical systems, one can think of the velocity being the right observers to determine the evolution of pressure. 
Therefore,  learning of differential and algebraic equations can be easily separated, and we can identify the underlying differential and algebraic equations using only data. 
%
In contrast to that, the projection-based approaches such as
POD need to respect the differential-algebraic structure of the Navier-Stokes equations, in
particular if there is inhomogeneity in the incompressibility constraint or
if the basis vectors are not divergence-free.
%
%
Furthermore, an affine-linear variant of DMD that is completely data-based
too, showed very similar performance with operator inference. However, as illustrated in a small example, we expect to see a limit of DMD when the nonlinear term plays a significant role in defining the dynamics. 
%